\documentclass[a4paper,11pt]{amsart}
\usepackage{amsfonts,mathrsfs}
\usepackage{amsthm}
\usepackage{amssymb}
\usepackage{enumerate}
\usepackage{tikz-cd}
\usepackage{MnSymbol}
\usepackage{color}
\usepackage{csquotes}
\usepackage{cleveref}

\theoremstyle{definition}
\newtheorem{Def}{Definition}[section]
\newtheorem{exm}[Def]{Example}
\newtheorem{Exm}[Def]{Example}
\newtheorem{rem}[Def]{Remark}
\newtheorem{Rem}[Def]{Remark}

\newtheorem{notation}[Def]{Notation}

\theoremstyle{plain}
\newtheorem{prop}[Def]{Proposition}
\newtheorem{Prop}[Def]{Proposition}
\newtheorem{thm}[Def]{Theorem}
\newtheorem{Thm}[Def]{Theorem}
\newtheorem*{thm*}{Theorem}
\newtheorem*{Thm*}{Theorem}
\newtheorem{lem}[Def]{Lemma}
\newtheorem{Lem}[Def]{Lemma}
\newtheorem{cor}[Def]{Corollary}
\newtheorem{Cor}[Def]{Corollary}
\newtheorem*{cor*}{Corollary}
\newtheorem*{Cor*}{Corollary}

\newtheorem*{con*}{Conjecture}

\newcommand\xqed[1]{%
  \leavevmode\unskip\penalty9999 \hbox{}\nobreak\hfill
  \quad\hbox{#1}}
\newcommand\EndExample{\xqed{$\Diamond$}}

\usepackage{mathrsfs}

\newcommand\R{\mathbb{R}}
\newcommand\C{\mathbb{C}}
\newcommand\Q{\mathbb{Q}}

\newcommand\N{\mathbb{N}}
\newcommand\A{\mathbb{A}}
\renewcommand\P{\mathbb{P}}

\newcommand\wt{\widetilde}
\newcommand\wh{\widehat}
\newcommand\ol{\overline}
\newcommand\scp[1]{\langle #1 \rangle}

\renewcommand\P{\mathbb{P}}

\newcommand{\conv}{\operatorname{conv}}
\newcommand{\Int}{\operatorname{int}}

\newcommand{\im}{\operatorname{im}}

\newcommand{\NC}{\mathcal{N}}
\newcommand{\cNC}{\mathcal{N}_+}
\newcommand{\CN}{{\rm CN}}
\newcommand{\clzar}{{\rm cl}_{\rm Zar}}
\newcommand{\bT}{\mathbb{T}}

\newcommand{\vnull}{\mathbf{0}}
\newcommand{\suchthat}{\ |\ }

\newcommand{\cV}{\mathcal{V}}
\newcommand{\id}{{\rm id}}
\renewcommand{\epsilon}{\varepsilon}
\renewcommand{\phi}{\varphi}

\title[{Families of faces of a convex semi-algebraic set}]{Families of faces and the normal cycle of a convex semi-algebraic set}
\author{Daniel Plaumann}
\address{Technische Universität Dortmund, Dortmund, Germany}
\email{daniel.plaumann@tu-dortmund.de}
\author{Rainer Sinn}
\address{Universität Leipzig, Leipzig, Germany} 
\email{rainer.sinn@uni-leipzig.de}
\author{Jannik Lennart Wesner}
\address{Technische Universität Dortmund, Dortmund, Germany}
\email{jannik.wesner@tu-dortmund.de}
\makeatletter
\@namedef{subjclassname@2020}{%
  \textup{2020} Mathematics Subject Classification}
\makeatother
\begin{document}


\begin{abstract}
  We study families of faces for convex semi-algebraic sets via the normal cycle which is a semi-algebraic set similar to the conormal variety in projective duality theory. We propose a convex algebraic notion of a \emph{patch} -- a term recently coined by Ciripoi, Kaihnsa, L\"ohne, and Sturmfels as a tool for approximating the convex hull of a semi-algebraic set. We discuss geometric consequences, both for the semi-algebraic and convex geometry of the families of faces, as well as variations of our definition and their consequences.
\end{abstract}
\maketitle

\section*{Introduction}
The main topic of our paper is the boundary structure of convex semi-algebraic sets. The faces in the boundary of such a set come in semi-algebraic families which cover the boundary. Our goals are first to make precise what the geometrically meaningful notions of families are in this setup and second to study the basic topological properties of the covering of the boundary. We propose a definition of a \emph{patch} and discuss its properties and variations, also with a view to computations.

There are many natural examples of convex semi-algebraic sets coming from a variety of different sources. When we think about the facial structure, we might first look at polytopes, which exhibit a finite number of faces in every intermediate dimension, the structure of which is encoded in a finite lattice. On the other end of the spectrum are convex bodies with a smooth, positively curved boundary made up entirely of extreme points. But many convex bodies in convex algebraic geometry and applications fall somewhere in between these two extremes. This is especially true for many examples arising as convex sets of matrices, like spectrahedra, or sets that are presented as the convex hull of some lower dimensional set. Spectrahedra (and more generally hyperbolicity cones) are central objects at the intersection of convex algebraic geometry and optimization \cite{zbMATH06125965} \cite{MR2198215}.
Convex hulls of real algebraic sets appear, for instance, in the context of attainable regions of dynamical systems \cite{patches} 
and in the study of quantum systems 
\cite{zbMATH07328163} and have been studied in classical projective geometry \cite{zbMATH06075449} \cite{zbMATH06045773}.

A computational approach to identify families of faces for convex hulls of curves was presented in \cite{patches} where the authors introduce the concept of a patch and give a numerical algorithm to compute the boundary structure. We propose a definition of a patch based on semi-algebraic geometry rather than using analytical notions and derive the main analytical features from our geometric definition. 

The fact that the boundary of a convex semi-algebraic set is covered by (semi-)algebraic families of faces follows essentially from duality theory in convex geometry combined with quantifier elimination (or, more generally, cylindrical algebraic decomposition). Our approach exploits these elements to define a \emph{patch} geometrically. A (primal) patch is a primal-dual object that encodes a (exposed) face as the set of points given by a supporting hyperplane and a family of faces by varying the supporting hyperplane in a connected semi-algebraic set, see \Cref{def:patch}. 
This all takes place in the \emph{normal cycle} from convex geometry -- a notion that is quite similar to conormal varieties from classical projective geometry. As a general reference for duality in convex algebraic geometry, in particular the interplay of convex duality and projective duality theory, see \cite{Sinn2015}. We discuss the relevant features of the normal cycle in \Cref{sec:normalcycle}.
We then turn our attention to geometric and topological properties of the covering of the boundary of the convex sets by the family of faces in a patch. The first general result is \Cref{dim_of_faces}, connecting patches to projective duality. The first step towards Hausdorff continuity of faces varying in a patch is \Cref{hausdorff_gen}. There is a technical issue with Hausdorff continuity of the family of \enquote{faces} in a patch that we illustrate by example (see \Cref{Example:Bellows} and \Cref{Example:DivingHelmet}). In fact, patches can cut faces into parts and produce a family of subsets of faces, and the subset in each face is not even necessarily connected.

The main reason for our distinction between patches (which are contained in the biregular locus of the appropriate conormal varieties) and their closures which we call closed patches is \Cref{dim_of_faces}: The dimension of the faces in the family corresponding to a patch is constant. However, by taking closure, we might get faces of higher dimension in the family corresponding to the closed patch (as illustrated for instance by the elliptope in \Cref{exm:elliptope}). 
The family of faces corresponding to a patch is Hausdorff continuous under the additional assumption that for every patch the interior of a face is either entirely contained in it or disjoint to it. This is proved in \Cref{sec:hyperbolic}, along with the fact that this assumption is always satisfied for hyperbolicity cones.

The normal cycle as a tool to study families of faces relies on duality, which is technically simpler and less prone to exceptions in the homogeneous setup, that is to say, for convex cones and projective varieties. On the other hand, some of our geometric intuition, as well as metric questions like convergence in the Hausdorff metric, are more easily phrased for compact convex sets. We will therefore adopt both points of view and go back and forth as needed.

\section{The Normal Cycle}\label{sec:normalcycle}
In this section, we discuss the basics of the normal cycle in convex geometry from our point of view motivated by convex algebraic geometry. For basics in convex geometry, we refer to \cite{MR1940576}. For basics in (projective) algebraic geometry, we refer to \cite{MR1416564}, and for semi-algebraic geometry to \cite{bcr}.

We denote by $(\R^n)^\ast$ the dual space of $\R^n$. For a linear functional $\ell\in(\R^n)^\ast$ and $x\in\R^n$, we write $\scp{\ell,x}=\ell(x)$.

Let $K\subset\R^n$ be a compact convex semi-algebraic set. We usually assume that the origin is an interior point of $K$. The (polar) dual of $K$,
denoted by $K^\circ$, is the convex set $\{\ell\in (\R^n)^\ast\suchthat \scp{\ell,x}\geq -1 \text{ for all }x\in K\}$. The polar dual of $K$ is semi-algebraic by quantifier-elimination. It is compact if the origin is an interior point of $K$. In fact, the separation theorem implies that $(K^\circ)^\circ = \overline{\conv(K\cup\{\vnull\})}$.

We use the same definition of the normal cycle as Ciripoi, Kaihnsa, L\"ohne, Sturmfels in \cite{patches}, namely
\[
  \NC(K) = \bigl\{(x,\ell)\in \partial K \times \partial K^\circ \suchthat \scp{\ell,x-x'} \leq 0 \text{ for all }x'\in K\bigr\} \subset \R^n\times \R^n.
\]
Thus the normal cycle consists of all pairs of points $(x,\ell)$, where $x$ is in the boundary of $K$ and $\ell$ corresponds to an inward normal vector of a supporting hyperplane to $K$ containing $x$. 
The normal cycle comes with the two projections $\pi_1$ and $\pi_2$ onto the first and second factor.

Usually, we work with convex cones, especially in proofs, and further pass from affine to projective space: By a \emph{proper cone} $C\subset\R^n$, we will mean a closed convex cone with the tip at the origin, with non-empty interior and not containing any lines. The latter condition is equivalent to $C\cap -C=\{\vnull\}$. The polar dual of a convex cone coincides with the usual notion of the dual cone $C^\vee=\{\ell\in(\R^n)^*\suchthat \ell(x) \ge 0\text{ for all }x\in C\}$. The dual of a proper cone is again proper and satisfies \emph{biduality} $(C^\vee)^\vee=C$.

Given a compact convex semi-algebraic set $K\subset \R^n$ with non-empty interior, we may homogenize and obtain the proper cone 
\[
  \wh{K} = \{ (\lambda x, \lambda)\suchthat x\in K, \lambda \geq 0\} \subset \R^n\times \R.
\]
Every proper cone can be obtained as such a homogenization with an appropriate choice of coordinates. 

We use $\P^n(\R)$ to denote $n$-dimensional real projective space and $\P^n(\R)^\ast$ for the dual projective space. For any proper cone $C$ in $\R^{n+1}$, we may take the image of $C\setminus\{\vnull\}$ in $\P^n(\R)$ under the canonical surjection $\R^{n+1}\setminus\{\vnull\}\to\P^n(\R)$ and obtain a semi-algebraic subset with non-empty interior in $\P^n(\R)$. We will usually not distinguish between $C$ as a subset of $\R^{n+1}$ and $C$ as a subset of $\P^n(\R)$, and likewise for the dual cone. Note that in passing to projective space in this way, the dimension of a semi-algebraic proper cone $C$ drops from $n+1$ to $n$.

\begin{Def}
  Let $C\subset\R^{n+1}$ be a proper cone. The \emph{(conic) normal cycle} of $C$ is the set 
  \[
   \cNC(C)=\bigl\{(x,\ell)\in\partial C\times\partial C^\vee \suchthat \ell(x) =0\bigr\}\subset\P^n(\R)\times(\P^n(\R))^\ast.
  \]
\end{Def}

\begin{Rem}\label{rem:normalcycles}
  The normal cycle of a convex body $K$ (i.e., a compact convex set containing the origin in its interior) and the conic normal cycle of its homogenization $\wh{K}$ are essentially the same. More precisely, we have
  $\cNC(\wh{K}) = \wh{\NC(K)}$, where $\wh{\NC(K)}$ is the image of $\NC(K)\subset \R^n\times \R^n$ under the embedding $\R^n\times \R^n \to \P^n(\R)\times (\P^n(\R))^*$, $(x,\ell)\mapsto ( (x:1),(\ell:1))$.
\end{Rem}

\begin{Lem}
    Let $C\subset \R^{n+1}$ be a proper convex cone. The normal cycle is self-dual, i.e., $\cNC(C) = \cNC(C^\vee)\subset \R^{n+1} \times (\R^{n+1})^*$ (after appropriate permutation of the factors).
\end{Lem}

\begin{proof}
    This follows from the biduality theorem in convex geometry, which, in our case, implies that $(C^\vee)^\vee = C$ holds  .
\end{proof}

Of course, the same holds for the normal cycle $\NC(K)$ of a compact convex set $K$ containing the origin in its interior (by \Cref{rem:normalcycles} or with the same proof using biduality for convex bodies).

Next, we show that the normal cycle $\NC_+(C)$ of a proper semi-algebraic cone $C\subset\R^{n+1}$, regarded as a subset of $\P^n(\R)\times \P^n(\R)^*$, has pure dimension $n-1$. To show this, we restrict to a compact base and show that the normal cycle $\NC(K)$ of a compact convex set containing the origin in its interior is semi-algebraically homeomorphic to the boundary of the Minkowski sum of $K$ with the unit ball $B(0,1) = \{x\in\R^n\suchthat \|x\|\leq 1\} \subset\R^n$.
\begin{notation}
  Let $K\subset \R^n$ be a compact convex semi-algebraic set containing the origin in its interior. Since $K$ is compact, the function $K\to \R$, $y\mapsto \|x-y\|$ achieves its minimum for every $x\in \R^n$. Moreover, the convexity of $K$ implies that this minimum is achieved at a unique point for every $x$. For these two functions, we fix the notation
  \begin{enumerate}
    \item $d_K \colon \R^n \to \R$, $x \mapsto \underset{y \in K}{\mathrm{min}} \|x-y\|$
    \item $p_K \colon \R^n \to K$, $x \mapsto \underset{y \in K}{\mathrm{argmin}} \, \|x-y\|$
  \end{enumerate}
  The map $p_K$ is called the \enquote{metric projection}.
\end{notation}

\begin{prop}\label{metric_proj}
  Let $K$ be a compact convex set containing the origin in its interior. The functions $d_K\colon \R^n \to \R$ and $p_K\colon \R^n \to K$ are semi-algebraic and continuous. Moreover, the function
  $u_K \colon \R^n \setminus K \to \partial K^\circ$, $x \mapsto \frac{x - p_K(x)}{\scp{p_K(x) - x, p_K(x)}}$ is well-defined, semi-algebraic, and continuous as well.
\end{prop}

\begin{proof}
  The map $d_K$ is semi-algebraic and continuous, see for example \cite[Proposition 2.2.8]{bcr} (by quantifier elimination and the triangle inequality).
  By \cite[Theorem~1.2.1]{schneider_2013}, the map $p_K$ is contracting and hence continuous. The map $p_K$ is semi-algebraic by quantifier elimination because its graph is the set $\{(x,y)\in (\R^n\setminus K) \times K \suchthat \forall \, z\in K\suchthat \|x-y\|\leq \|x-z\|\}$.
  
  We simply write $d$ and $p$ for $d_K$ and $p_K$ for the remainder of the proof.

  Clearly $p(x) \in \partial K$ for any $x \in \R^n \setminus K$. It remains to show that for every $x \in \R^n \setminus K$ we have
  \[ \scp{p(x)-x,p(x)} \not= 0 \; \text{ and } \; \frac{x - p(x)}{\scp{p(x) - x, p(x)}} \in \partial K^\circ .\]

  We first show that
  \[ \forall y \in K: \; \scp{x-p(x),y} \leq \scp{x-p(x),p(x)}. \]
  Let $y \in K$. Set $f(t) := \|x - \big(p(x)+t(y-p(x))\big)\|^2$. Since $f$ has a minimum at $0$ on $[0,1]$, we have
  \[ 0 \leq f'(0) = 2 \scp{x-p(x),p(x)-y}. \]

  Since the origin is an interior point of $K$, we can choose a point $y\in K$ such that $\scp{x-p(x),y} > 0$. This implies that $\scp{p(x)-x,p(x)} < 0$ for all $x\in \R^n\setminus K$. 
  Computing $\scp{u_K(x),p(x)} = -1$ shows that $\frac{x - p(x)}{\scp{p(x) - x, p(x)}}$ lies in the boundary of $K^\circ$ as claimed.

  The map $u_K$ is continuous and semi-algebraic as the composition of continuous and semi-algebraic functions.
\end{proof}

\begin{lem}\label{boundaryK1}
  Let $K\subset \R^n$ be a compact semi-algebraic set containing the origin in its interior.
  The set $d_K^{-1}([0,1])$ is the Minkowski sum of $K$ and the unit ball $B(0,1)$ and therefore a compact convex semi-algebraic set containing $K$. Its boundary is the preimage of $1$ under $d_K$.
\end{lem}

\begin{proof}
  Write $K_1$ for $d_K^{-1}([0,1])$. It is semi-algebraic as the preimage of a semi-algebraic set with respect to a semi-algebraic function. 
  The set $K_1$ contains $K = d_K^{-1}(\{0\})$. It is convex because $d_K$ is a convex function by the triangle inequality. It is closed because $d_K$ is continuous and it is bounded and hence compact. 

  So we only have to show that $\partial K_1 \supset d_K^{-1}(\{1\})$, as we get $\partial K_1 \subset d_K^{-1}(\{1\})$ from continuity.
  Let $x \in d_K^{-1}(\{1\})$ and $y \in K$. We consider the function $g(t) = \| y - \big(x+t(x-p(x))\big)\|^2$.
  We saw in the proof of \Cref{metric_proj} that $\scp{x-p(x),y} \leq \scp{x-p(x),p(x)}$ holds for all $y\in K$.
  Hence
  \begin{align*}
    g'(0) &= \scp{y-x,p(x)-x}\\
      &= \scp{y-p(x),p(x)-x} + \scp{p(x)-x,p(x)-x}\\
      &\geq 0 + \|p(x)-x\|^2 = 1.
  \end{align*}
  Therefore $d_K(x+t(x-p(x))>1$ for all sufficiently small $t>0$, which shows that $x$ is in the boundary of $K_1$. 
\end{proof}

\begin{Thm}
  Let $K\subset\R^n$ be a compact semi-algebraic convex set containing the origin in its interior and let $K_1$ be the Minkowski sum of $K$ and the unit ball $B(0,1)$. Let $p_K$ be the metric projection onto $K$ and let $u_K$ be the map $u_K \colon \R^n \setminus K \to \partial K^\circ$, $x \mapsto \frac{x - p_K(x)}{\scp{p_K(x) - x, p_K(x)}}$
  The map
  \[ \phi \colon \partial K_1 \longrightarrow \NC(K), \quad x \longmapsto (p_K(x),u_K(x)) \]
  is a semi-algebraic homeomorphism. In particular, the normal cycle $\NC(K)$ is a compact semi-algebraic set of pure dimension $n-1$.
\end{Thm}
\begin{proof}
  The map $\phi$ is a semi-algebraic and continuous function.
  Let
  \[ \psi \colon \NC(K) \longrightarrow \partial K_1, \quad (x,\ell) \longmapsto x - \frac{\ell}{\| \ell \|}\]
  The map $\psi$ is well defined by \Cref{boundaryK1} because $\psi(x,\ell)$ has distance $1$ from $x$ and $\| y - \psi(x,\ell)\|^2 \geq 1$ for all $y\in K$. 
  It is clearly semi-algebraic and continuous.
  For every $y\in K$
  \begin{align*}
    \| \psi(x,\ell) - y\|^2 &= \lVert y \rVert^2 - 2 \scp{y, \psi(x,\ell)} + \lVert \psi(x,\ell) \rVert^2\\
      &= \lVert y \rVert^2 - 2 \scp{y,x} + \frac{2}{\lVert \ell \rVert} \underbrace{\scp{y, \ell}}_{\geq -1} + \lVert x \rVert^2 - \frac{2}{\lVert \ell \rVert}
        \underbrace{\scp{x,\ell}}_{= -1} + 1\\
      &\geq \lVert y \rVert^2 - 2 \scp{y,x} + \lVert x \rVert^2 + 1 = \|x-y\|^2 + 1,
  \end{align*}
  showing that $p_K(\psi(x,\ell)) = x$.
  Now
  \[ u_K(\psi(x,\ell)) = \frac{x - \frac{\ell}{\lVert \ell \rVert} - x}{\scp{x - x + \frac{\ell}{\lVert \ell \rVert}, x}} = - \frac{\ell}{\scp{\ell,x}} = \ell .\]
  Therefore $\phi \circ \psi = \mathrm{id}_{\NC(K)}$. Next
  \begin{align*}
    \psi(\phi(x)) &= p_K(x) - \frac{u_K(x)}{\lVert u_K(x) \rVert}\\
      &= p_K(x) - \frac{(x - p_K(x))\lvert \scp{p_K(x) - x, p_K(x)}\rvert}{\underbrace{\scp{p_K(x)-x,p_K(x)}}_{<0\text{ by proof of } \ref{metric_proj}} \lVert x - p_K(x) \rVert}\\
      &= p_K(x) + \frac{x-p_K(x)}{\underbrace{\lVert x - p_K(x) \rVert}_{=1}} = x .
  \end{align*}
  Hence $\psi \circ \phi = \mathrm{id}_{\partial K_1}$. Therefore $\psi = \phi^{-1}$ and we are done.
\end{proof}

\section{The conormal variety and patches}

We use $\P^n$ to denote complex projective $n$-space, containing real projective space $\P^n(\R)$ as a subset. A \emph{real projective variety} is a projective variety defined over $\R$. Such a variety is the common zero set $\cV(f_1,\dots,f_r)$ of real homogeneous polynomials $f_1,\dots,f_r$ in $n+1$ variables. For a real projective variety $X$, we denote the \emph{real locus} of $X$ by $X(\R)$ and the \emph{regular locus} of $X$ by $X_{\rm reg}$. The (projective) tangent space of $X$ at a regular point $x\in X_{\rm reg}$, regarded as a linear subspace of $\P^n$, is denoted by $\bT_x(X)$. 

For a subset $S$ of $\P^n$, we denote the closure of $S$ in the (real) Zariski topology by $\clzar(S)$, i.e., the smallest (real) projective variety containing $S$. The notation $\ol S$ is reserved for the euclidean closure of $S$ in real or complex projective space. 

\begin{Def}
  Let $S\subset\P^n(\R)$ be a semi-algebraic set. The \emph{algebraic boundary} of $S$ is the Zariski-closure of the euclidean boundary $\partial S$ and is denoted by $\partial_a S$. 
\end{Def}

The algebraic boundary of $S$ is therefore a projective variety in $\P^n$, typically with both real and complex points. 
For any open semi-algebraic subset $S$ of $\P^n(\R)$, the algebraic boundary is a real hypersurface, i.e., it is of pure algebraic dimension $n-1$. By our usual abuse of notation, the algebraic boundary $\partial_a C$ of a proper cone $C\subset\R^{n+1}$ is the algebraic boundary of $C$ regarded as a subset of $\P^n(\R)$. As $C$ is the closure of its interior, its algebraic boundary is thus a hypersurface, possibly with several irreducible components.

\begin{Def}
  Let $X$ be a real projective variety embedded into $\P^n$. The \emph{conormal variety} of $X$ is the projective variety
  \[
    \CN(X)=\clzar\bigl(\bigl\{(x,\ell)\in\P^n\times(\P^n)^\ast\suchthat x\in X_{\rm reg}\text{ and }\ell\equiv 0\text{ on }\bT_x(X)\bigr\}\bigl).
  \]
  The image of $\CN(X)$ under the projection onto the second factor is called the \emph{projective dual} of $X$ and is denoted by $X^\ast$. We will refer to the open subvariety
  \[
    \CN_{\rm bireg}(X)=\bigl\{(x,\ell)\in\CN(X)\suchthat x\in X_{\rm reg}\text{ and }\ell\in (X^\ast)_{\rm reg}\bigr\}  
  \]
  as the \emph{biregular locus} of $\CN(X)$. 
\end{Def}

\begin{rem}\label{conormalhypersurface}
  If $X \subset \P^n$ is an irreducible hypersurface defined by the vanishing of an irreducible real homogeneous polynomial $f$, and if $x$ is a regular point of $X$, then $(x,\nabla f(x))\in \CN(X)$ and moreover $\pi_1^{-1}(x) = \{(x,\nabla f(x))\}$, where $\pi_1$ is the projection of $\CN(X) \subset \P^n\times \P^n$ onto the first factor.
\end{rem}

\begin{Def}\label{def:patch}
  Let $C\subset\R^{n+1}$ be a proper cone and let $Y$ be an irreducible component of $\partial_a C$. A \emph{(primal) patch} of $C$ (over $Y$) is a connected component of the semi-algebraic subset of $\cNC(C)\cap\CN_{\rm bireg}(Y)$ consisting of pairs $(x,\ell)$ such that $x$ is not contained in any irreducible component of $\partial_a C$ other than $Y$.
  The euclidean closure of a patch (in $\P^n(\R)\times\P^n(\R)^\ast$) is called a \emph{closed (primal) patch} of $C$ (over $Y$). The finite family of all primal patches of $C$ is denoted by $\mathscr{P}(C)$. A \emph{dual patch} of $C$ is a patch of the dual cone of $C$. 
\end{Def}

While this definition certainly looks complicated, we will now examine a series of examples that should help to visualize it and to explain why we settled on these exact terms. Our examples will be convex bodies $K\subset \R^n$, regarded as affine sections of their homogenizations $\wh{K} = \{(\lambda x, \lambda)\in\R^n\times \R\suchthat \lambda\geq 0, x\in K\}$, as explained in the previous section. As the normal cycle $\NC(K)$ of $K$ viewed as a subset of $\P^n(\R)\times \P^n(\R)^*$ is the same as the normal cycle $\NC_+(\wh{K})$ (see \Cref{rem:normalcycles}), we can interpret the homogeneous \Cref{def:patch} of a patch for $\wh{K}$ in the affine chart containing $K$.
\begin{rem}\label{affinesetup}
  Let $K\subset \R^n$ be a compact convex set containing the origin in its interior and let $X$ be an irreducible component of the algebraic boundary of $K$. We write $C = \wh{K}$ for its homogenization, which is a cone in $\R^n\times \R$. We saw in \Cref{rem:normalcycles} that the normal cycle $\NC(K)$ of $K$ and the projective normal cycle $\cNC(C)$ of $C$ are semi-algebraically homeomorphic. The algebraic cone over $X$ embedded into $\C^n\times \C$ as $\{(x,1)\suchthat x\in X\}$ is an irreducible component of the algebraic boundary of $C$ if we interpret it as a subvariety $Y$ of $\P^n$. As such, $Y$ is the projective closure of $X$ with respect to the embedding $\A^n \to \P^n$, $x\mapsto (x:1)$. So by a (primal) patch of $K$ (over $X$) we mean a connected component of the semi-algebraic set $\cNC(C)\cap \CN_{\rm bireg}(Y)$ identified with a subset of $\NC(K)$ via the semi-algebraic homeomorphism in \Cref{rem:normalcycles}.
\end{rem}

\begin{exm}[Circle and a stick]
\label{Example:Bellows}
  Let $K \subset \R^3$ be the convex hull of the following set
  \[ \{ (-1,0,-1), (-1,0,1)\} \cup \{ (x,y,0) \in \R^3 \suchthat x^2+y^2=1\},\]
  a unit circle in the $(x,y)$-plane and an interval parallel to the $z$-axis.
  Then $K$ is a convex body, its algebraic boundary consists of two quadratic cones
  \[ (x-z)^2 +y^2 - (z+1)^2, \quad (x+z)^2+y^2-(z-1)^2 .\]
  The singular locus of $\partial_a K$ in $\partial K$ is the circle and the \enquote{stick} \[{\rm conv} \{(-1,0,-1),(-1,0,1)\}\] attached to it.
  The dual convex body $K^\circ$ is a \enquote{bellows}, it is the convex hull of the two ovals
  \begin{align*}
    \{ (x,y,x-1) \in \R^3 \suchthat x^2+y^2 = 1\} &\text{ in the hyperplane } \{\ell \suchthat \scp{(-1,0,1),\ell} = -1 \}\\
    \{ (x,y,1-x) \in \R^3 \suchthat x^2+y^2=1\} &\text{ in the hyperplane } \{\ell \suchthat \scp{(-1,0,-1),\ell} = -1 \}
  \end{align*}
  which meet in the point $(1,0,0)$.
  \begin{figure}
    \begin{center}
      \begin{minipage}{.48\textwidth}
      \includegraphics[width=6cm]{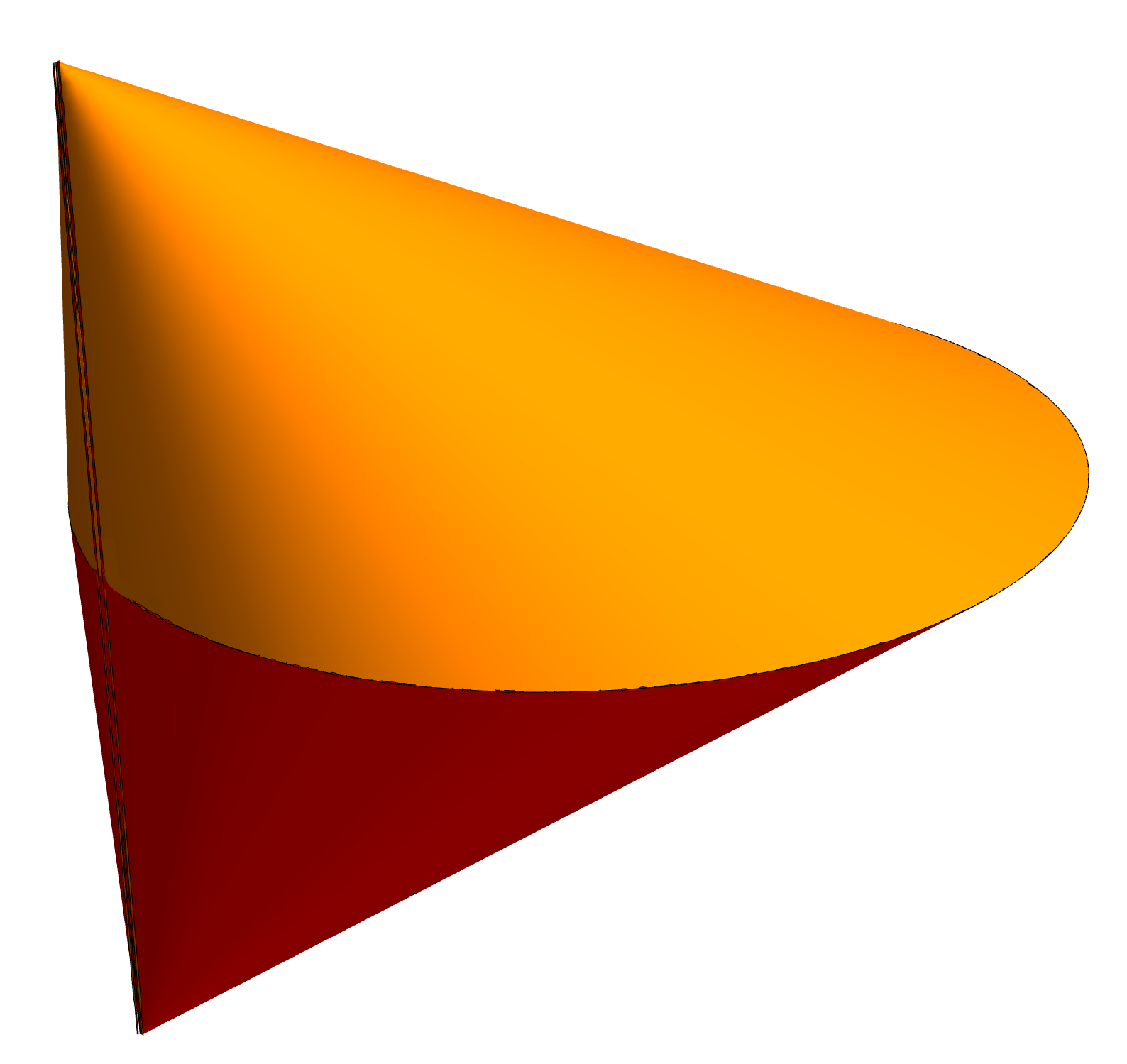}
      \end{minipage}\hspace*{1cm}
      \begin{minipage}{.4\textwidth}
        \includegraphics[width=3.5cm]{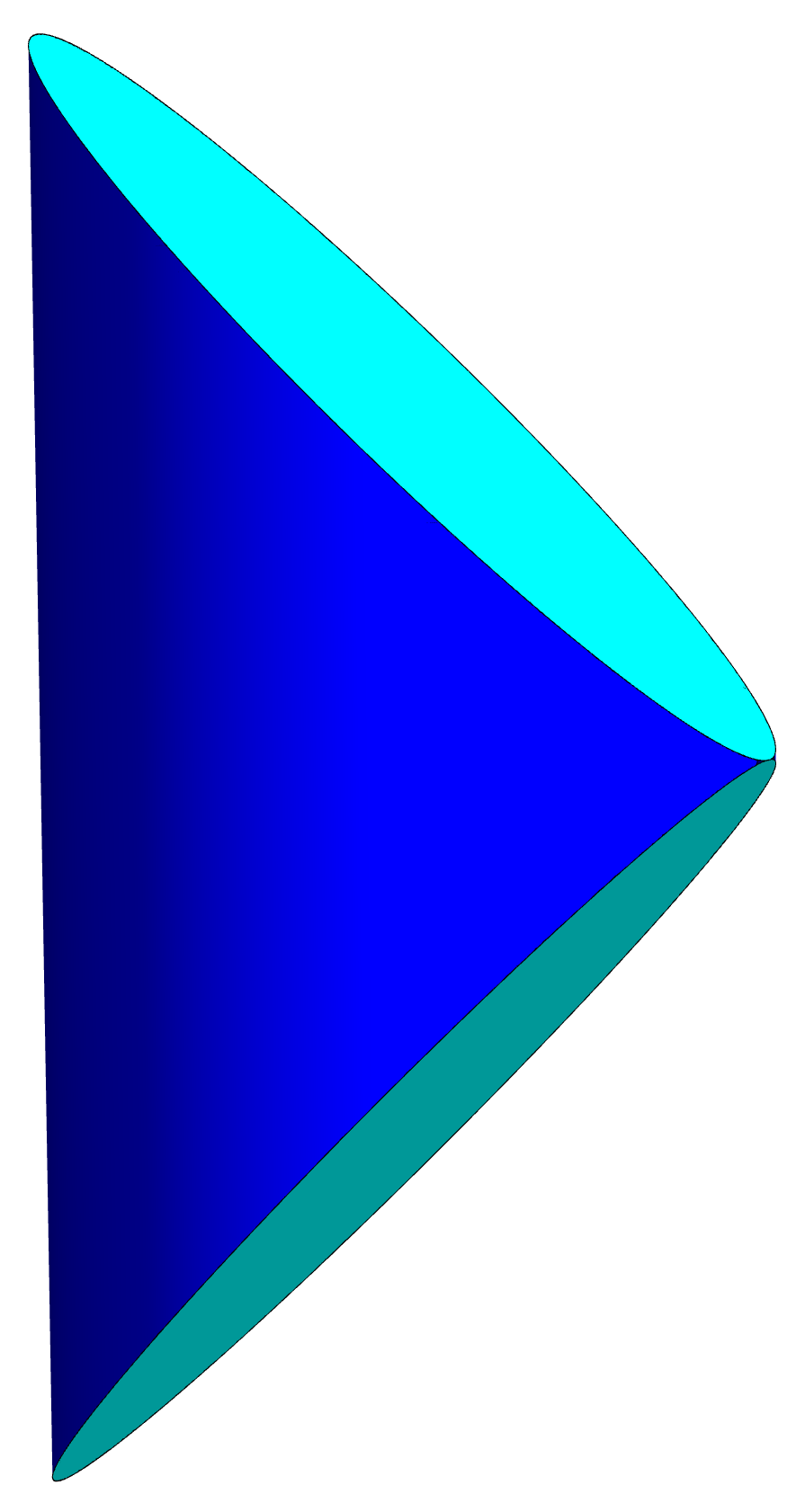}
      \end{minipage}
      \caption{The convex body $K$ and its dual body $K^\circ$ from Example \ref{Example:Bellows}}
    \end{center}
  \end{figure}
  
  There are two primal patches of $K$, one for each quadratic cone. The patch for each cone contains the points in the intersection of the cone with the boundary of $K$ that are neither on the circle nor the stick, together with the unique supporting hyperplane at each such point, which is equal to the tangent hyperplane to the quadratic cone.
  The closed patches also contain the circle and the upper or respectively the lower half of the stick with the appropriate supporting hyperplane at these points.
  
  Note that both cones are smooth at the points in the relative interior of the stick. By demanding in our definition of patch that the points $x$ lie on only one irreducible component of $\partial_a K$, we ensure pure dimensionality of the closed patch and that the protruding half of the stick is excluded. By taking the closure, only a subset of the face that is the stick is added, namely the part that makes it continuous in the Hausdorff metric.
  
  While the condition that a point $x\in \partial K$ lies on only one irreducible component of $\partial_a K$ is sufficient for the boundary of $K$ to locally coincide with one of the irreducible components of $\partial_a K$,
  it is not essential.
  Let $Y$ be an irreducible component of $\partial_a C$.
  Then a necessary and sufficient condition for $\partial K$ to coincide locally with $Y$ at a $x \in \partial K \cap Y_{\rm reg}$ would be that $Y$ is the only 
  irreducible component of the algebraic boundary such that
  its intersection with $\partial K$ has local dimension $n-1$ at $x$. 
  This condition also distinguishes the two halves of the stick and puts them in separate patches. Apart from the algebraic structure of the boundary, this separation of the stick into two halves also reflects the convex geometric fact that the stick is not the Hausdorff limit of either family of edges.\EndExample
\end{exm}

Most of the theory below could be developed with the property \enquote{only one irreducible component of $\partial_a K$ has local dimension $n-1$ at $x$} substituting the condition that \enquote{$x$ lies on only one irreducible component of $\partial_a K$}. In some ways, this might even be considered more intuitive. On the other hand, it is a much more difficult condition to check in practice, because it is semi-algebraic in nature. We have therefore decided to choose the algebraic, second condition in this paper for the definition of a patch.
To further illustrate this point, let us take a look at a further example.

\begin{exm}[Circle and three tangents]
  Let $K \subset \R^2$ be the convex hull of the unit circle, centered at the origin, and the points $(-1,-1)$ and $(-1,1)$.
  
  The components of the algebraic boundary are given by the circle $x^2+y^2-1$ and three tangent lines $y+1$, $y-1$ and $x+1$ to the circle.
  The singular locus of $\partial_a K$ in $\partial K$ consists of the points of tangency of the lines to the circle: 
  \[ \begin{pmatrix} -1\\ 0 \end{pmatrix},  \begin{pmatrix} 0\\ -1 \end{pmatrix},  \begin{pmatrix} 0\\ 1 \end{pmatrix}. \]
  The primal patches of $K$ are
  \begin{align*}
    \{ ((t,-1),(0,1)) \in \R^2\times \R^2 \suchthat -1 < t < 0 \}\\
    \{ ((t,1),(0,-1)) \in \R^2\times \R^2 \suchthat -1 < t < 0 \}\\
    \{ ((-1,t),(1,0)) \in \R^2\times \R^2 \suchthat -1 < t < 0 \}\\
    \{ ((-1,t),(1,0)) \in \R^2\times \R^2 \suchthat 0 < t < 1 \}\\
    \{ ((x,y),\ell) \in \R^2\times \R^2 \suchthat x^2+y^2 =1 , x > 0,\\
     \ell \text{ unit normal of the tangent line at } (x,y) \}
  \end{align*}
  Note that the line $\{x=-1\}$ and $\partial K$ coincide locally at $(-1,0)$.
  While $(-1,0)$ also lies in the unit circle, its intersection with $\partial K$ has local dimension $0$ at $(-1,0)$. 

  So here, the condition that $x$ should only lie on one irreducible component in order for $(x,\ell)$ to be in a patch reflects the algebraic structure of the boundary, rather than the convex geometric structure.\EndExample
\end{exm}
  
The next example motivates our usage of the biregular locus, namely the desire for a patch to project into the regular locus of the dual variety.
Using biduality we show that a patch is a parameterized system of exposed faces of a fixed dimension. If we were to drop the requirement of biregularity, the dimension of the faces might vary.

\begin{exm}[The elliptope]
\label{exm:elliptope}
  Let $K \subset \R^3$ be the connected component of $\vnull$ in $\{ (x,y,z) \in \R^3 \suchthat f(x,y,z) \leq 0 \}$
  where $f = x^2 + y^2 + z^2 - 2x y z - 1$ is the Cayley cubic so that $K$ is an \enquote{inflated tetrahedron}.
  The set $K$ is a convex body with algebraic boundary given by $f$, an affine part
  of Cayley's nodal cubic surface.
  The singular locus of $\partial_a K$ in $\partial K$ consists of the four points
  \[ a=(1,1,1), b=(1,-1,-1), c=(-1,1,-1), d=(-1,-1,1) .\]
  Intersecting the four half-spaces
  \[
    \{\ell \suchthat \scp{x,\ell} \leq 3\}, \quad x \in \{a,b,c,d\}\\
  \]
  with $\{f\leq0\}$ gives us $K$.
  For every pair of distinct points $v,w \in \{a,b,c,d\}$ the edge ${\rm conv}\{v,w\}$ is an exposed face of $K$.
  Every other proper face of $K$ is an exposed point. The dual convex body $K^\circ$ is the convex hull of an affine part
  of the Roman surface, which is given by $g = x^2 y^2 + y^2 z^2 + x^2 z^2 - 2 x y z$.
  When moving on $K$ towards the interior of one of the edges the face dimension jumps from $0$ to $1$ despite
  the interior of the edge being regular in $\partial_a K$. 
  But the singular locus of $g$ in $\partial K^\circ$ are the pinch points of the roman surface, which expose the six edges
  of $K$ and hence the edges of $K$ are not included in any patch.
  
  The primal patches of $K$ are in fact the four areas resembling the interiors of the sides of a tetrahedron.
  Each is parameterized by the portion of the respective ``lobe'' of the roman surface, which lies in $\partial K^\circ$.

  \begin{figure}
    \begin{center}
      \begin{minipage}{0.45\textwidth}
        \includegraphics[width=5.5cm]{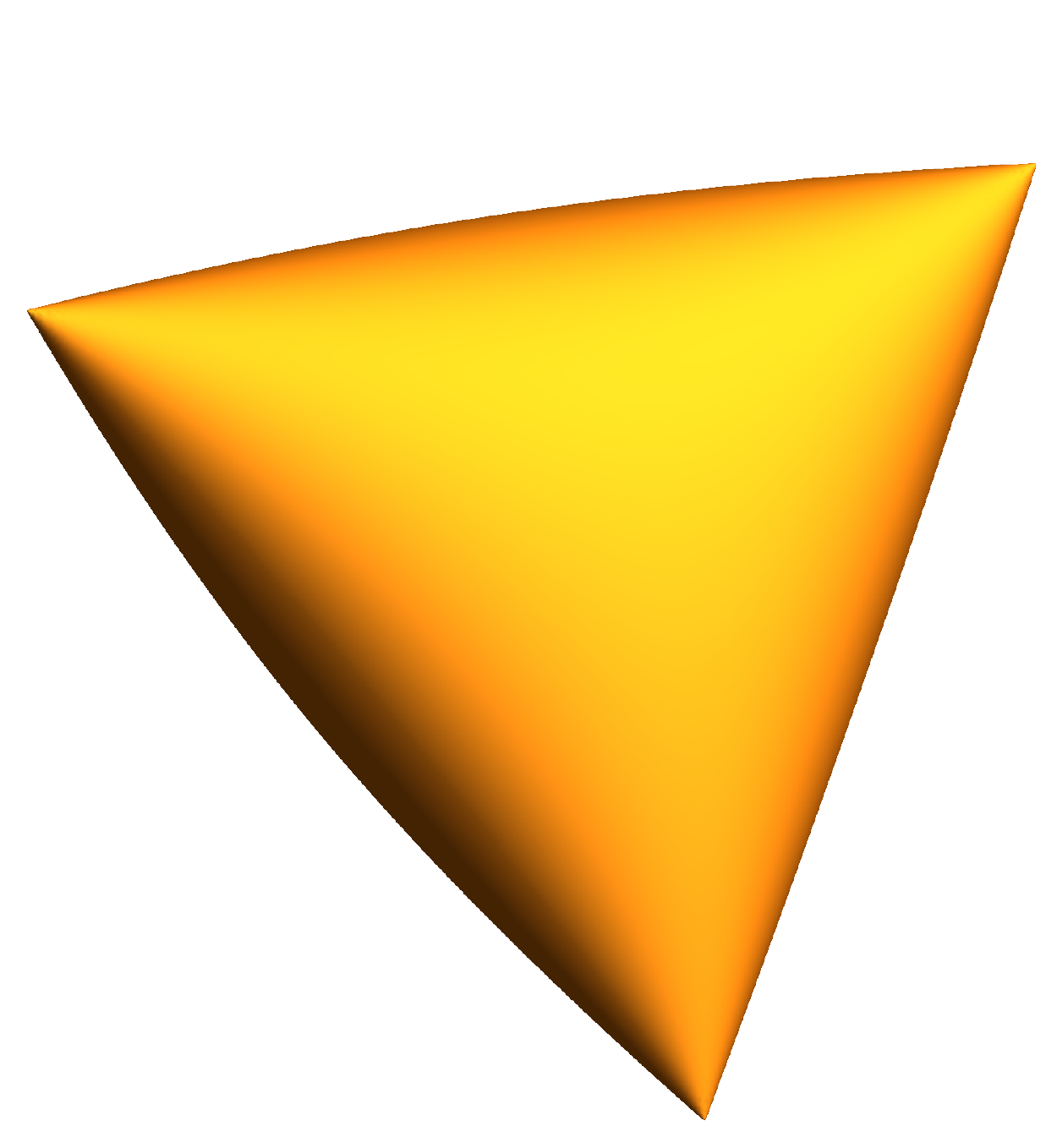}
      \end{minipage}\hspace*{2em}
      \begin{minipage}{0.45\textwidth}      
        \includegraphics[width=5.5cm]{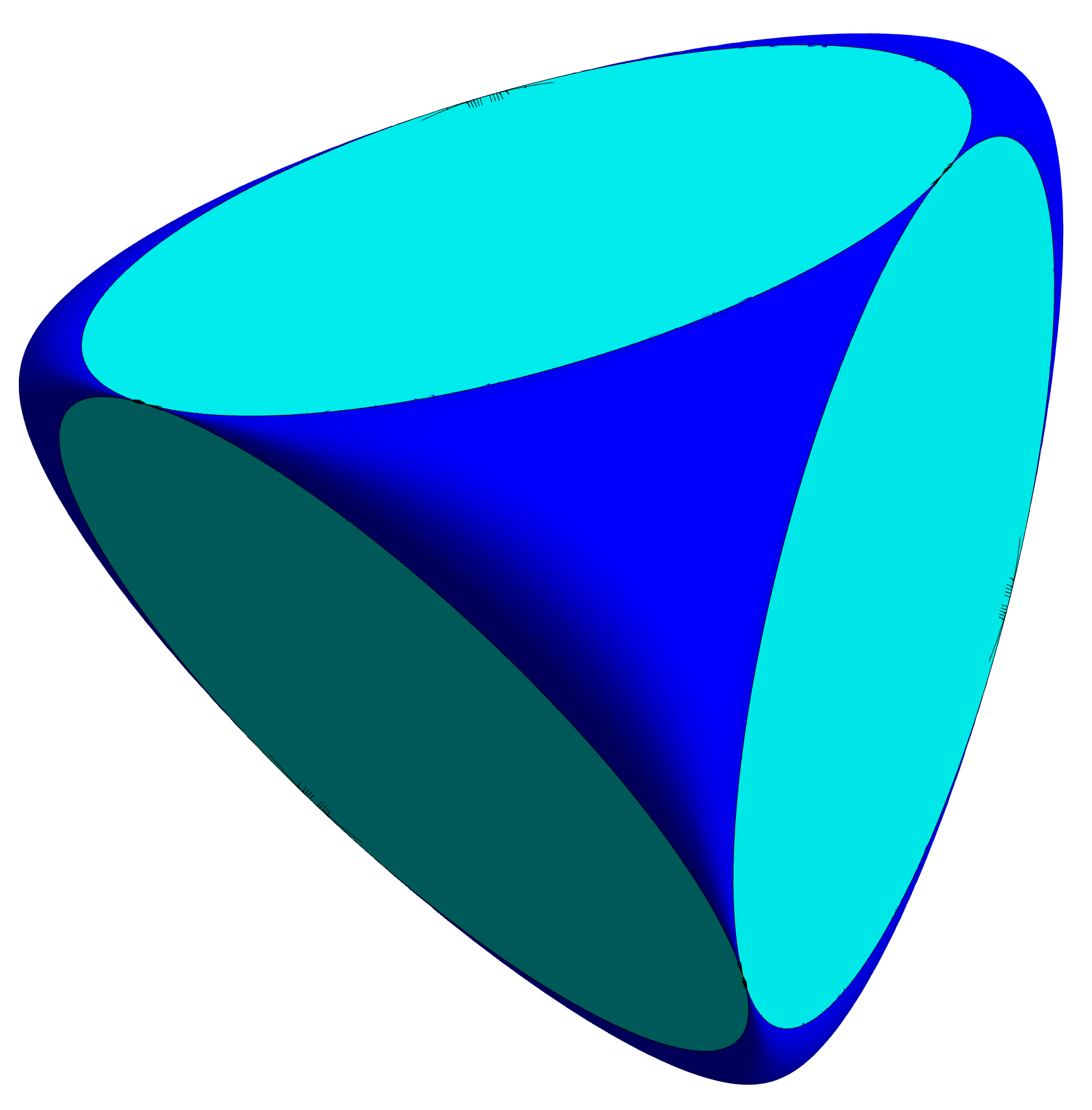}
      \end{minipage}
      \caption{The elliptope and its dual body from Example \ref{exm:elliptope}. The part of the dual boundary in blue is contained in the Roman surface, while the green part is added by taking the convex hull.}
    \end{center}
  \end{figure}

  If we dropped the condition of biregularity in the definition of a patch and rather considered the connected components of the set $\{(x,\ell)\in \CN(X)\suchthat x\in (\partial_a K)_{\rm reg}\}$, then we would get only one patch in this example of the Cayley cubic and the dimension of the faces is then not constant over a patch. Indeed, most points in the boundary of $K$ are exposed points but some points lie on the edges joining the four vertices $a,b,c,d$ of the tetrahedron. The biregularity condition ensures that the face dimension is constant, as we will see below.\EndExample
\end{exm}
  
Even though the closed primal patches cover $\partial K$ and the closed dual patches cover $\partial K^\circ$, the union of all closed patches --- primal and dual --- may not be dense in the normal cycle.
  
\begin{exm}[Cylinder over nodal cubic]
\label{Example:NodalCubic}
  For example consider the convex region $K$ bounded by the plane cubic $y^2 - (x+1)(x-1)^2 = 0$ and take the cylinder $K\times [-1,1]\subset \R^3$ over it.
  The dual convex set is the bipyramid over the dual set $K^\circ$, which is bounded by a quartic and a line. The $1$-dimensional face of $K^\circ$ is dual to the
  node of the cubic curve. So the $1$-dimensional face of the bipyramid over $K^\circ$ is dual to the $1$-dimensional face $(1,0)\times [-1,1]$ of $K\times [-1,1]$.
  The product of these two $1$-dimensional faces is a $2$-dimensional subset of the normal cycle of $K$ that is not in the closure of the union of all closed patches.\EndExample
  \begin{figure}
    \begin{center}
      \includegraphics[width=3.7cm]{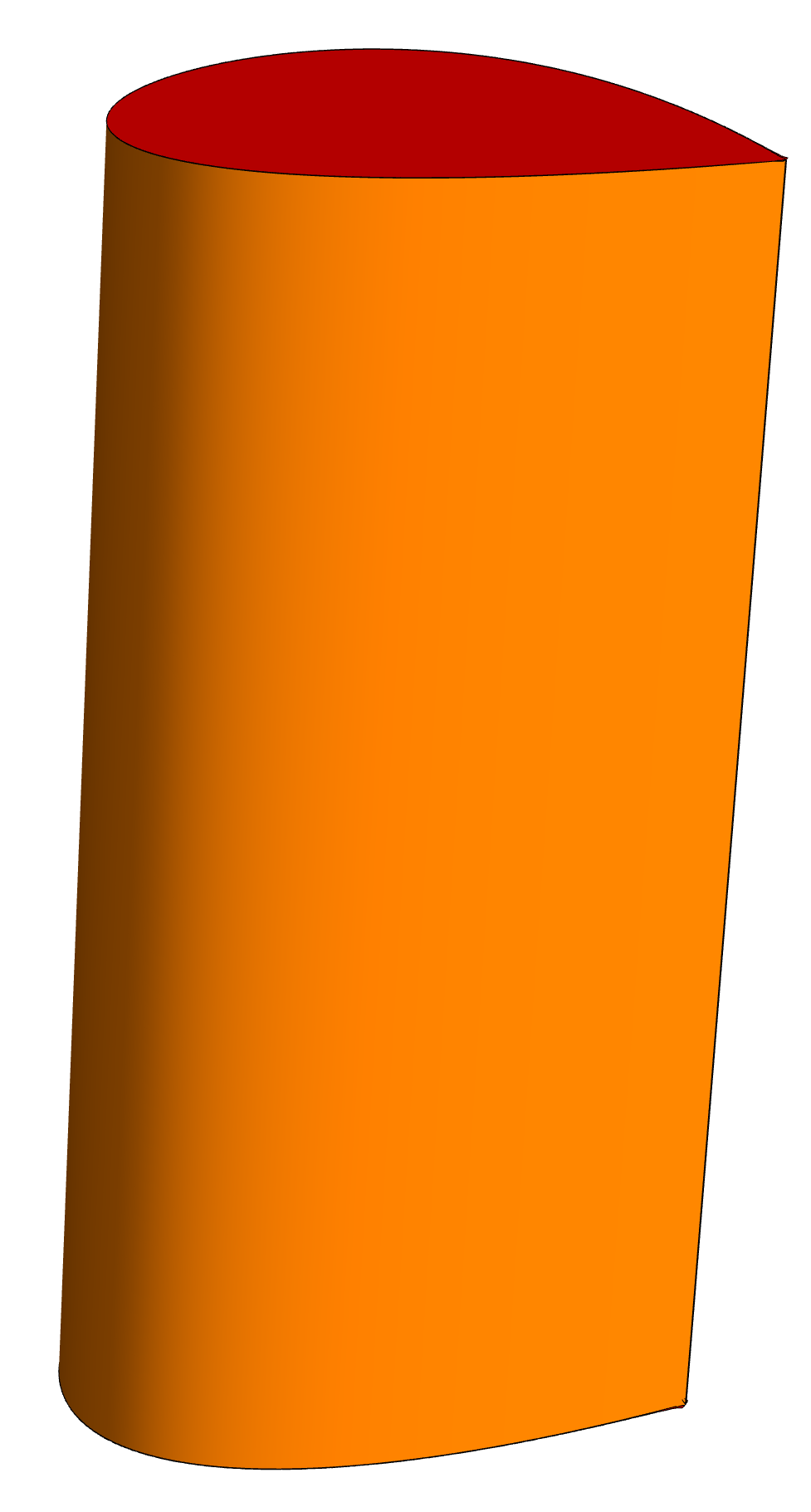}\hspace*{2em}
      \includegraphics[width=4cm]{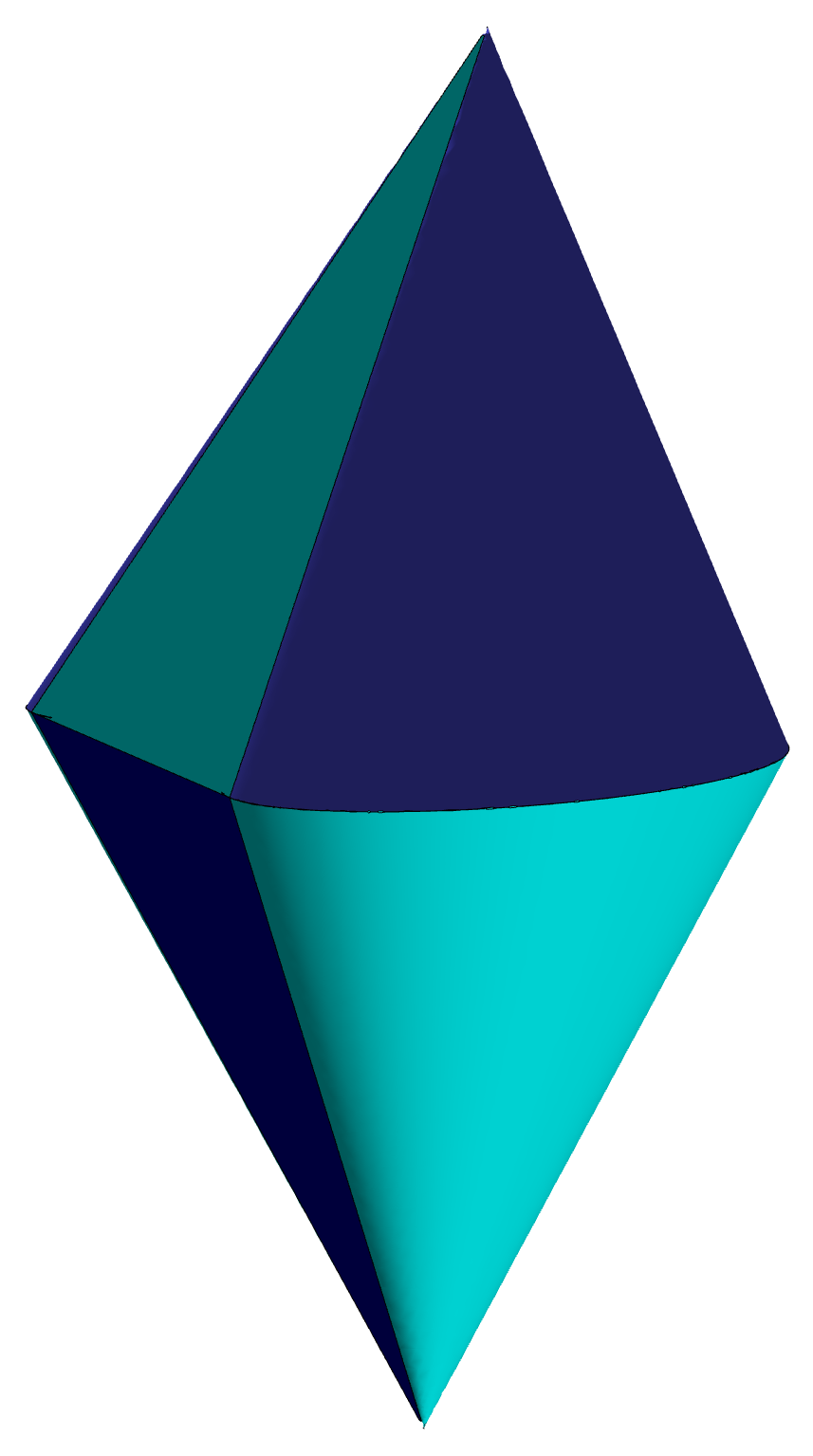}
      \caption{The convex body and its dual from Example \ref{Example:NodalCubic}}
    \end{center}
  \end{figure}
\end{exm}

\begin{rem}\label{bireg_dense}
  Let $X \subset \P^n$ be an irreducible projective variety. The restriction of the first projection $ \pi_1 : \CN(X)  \rightarrow X$ to $\pi_1^{-1}(X_{\rm reg})$ is an open map in the Zariski topology,
  since $\pi_1^{-1}(X_{\rm reg})$ is a complex vector bundle over $X_{\rm reg}$ and therefore locally trivial.
  In particular, $X_{\rm bireg} := \pi_1(\CN_{\rm bireg}(X))$ is open in $X$. Using the biduality theorem (see e.g.~\cite[Theorem~4.4.6]{flenner1999}) we get the same conclusion for the dual variety, i.e.,
  $X^*_{\rm bireg} := \pi_2(\CN_{\rm bireg}(X))$ is open in $X^*$. 
  Since $\CN_{\rm bireg}(X)$ is dense in $\CN(X)$, both $X_{\rm bireg}$ as a subset of $X$ and $X^*_{\rm bireg}$ as a subset of $X^*$ are dense.
\end{rem}

The boundary of a convex cone is covered by its closed patches:
\begin{Lem}\label{patches_cover}
  For any proper convex cone $C$, we have
  \[ \partial C = \bigcup_{P \in \mathscr{P}(C)} \pi_1(\ol P) \]
\end{Lem}

\begin{proof}
  The boundary of $C$ is covered by the intersections $X\cap \partial C$, where $X$ ranges over the irreducible components of the algebraic boundary $\partial_a C$. Applying \Cref{bireg_dense} to the intersections $X\cap \partial C$ implies the claim because every boundary point of $C$ has full local dimension in the semi-algebraic set $X\cap \partial C$ for some irreducible component $X\subset \partial_a C$.
\end{proof}

\begin{prop}
  Let $C$ be a proper convex cone, then
  \[ C^\vee = {\rm conv}\;\; \bigcup_{P \in \mathscr{P}(C)} \pi_2(\ol P). \]
\end{prop}
\begin{proof}
  Denote the set on the right hand side as $B$ which is clearly a convex cone and contained in $C^\vee$. Since $C^\vee$ is proper, $B$ is also closed. The other inclusion follows from duality.
  If $B \subsetneqq C^\vee$ then $C = (C^\vee)^\vee \subsetneqq B^\vee$ by duality.
  Hence there exists $x \in \partial C \cap \Int B^\vee$.
  Let $U \subset B^\vee$ be an open set with $x \in U$. Since $\bigcup_{P \in \mathscr{P}(C)} \pi_1(P)$ is euclidean-dense in $\partial C$ by \Cref{patches_cover},
  there is a $(y,\ell) \in \bigcup_{P \in \mathscr{P}(C)} P$ with $y \in U$. Furthermore, $y- \epsilon \ell \in U$ holds for a sufficiently small $\epsilon >0 $.
  Hence $\scp{\ell, y- \epsilon \ell } < 0$, but $\ell \in B$, a contradiction.
\end{proof}

\begin{Lem}
  Let $A\subset \R^n$ be a closed semi-algebraic set. Let $x$ be a point in the boundary of the interior of $A$.
  Assume that $x$ is a regular point on an irreducible component $Y$ of $\partial_a A$ and, moreover, that there exists an open neighborhood $U$ of $x$ in $\R^{n}$ such that $\partial A\cap U\subset Y(\R)\cap U$. Then there exists an open subset $V\subset U$ for which $x \in V$ and
  \[
    \partial A\cap V= Y(\R)\cap V.
  \]
\end{Lem}

\noindent The following proof is a careful application of the implicit function theorem.
\begin{proof}
  Let $\mathcal{I}(Y) = \langle f \rangle$.
  We can assume, by shrinking the neighborhood $U$ if necessary, that $Y \cap U \subset Y_{\rm reg}$ and $\scp{\nabla f(x), \nabla f(y)} \not= 0$ for all $y \in Y \cap U$.

  By the implicit function theorem (\cite[Corollary~2.9.8]{bcr}) there exists an open semi-algebraic set $V_1 \subset T := x +T_x Y$
  with $x \in V_1$ and an open semi-algebraic set $V_2 \subset \R$ with $0 \in V_2$ such that
  $V := \{ y + \lambda \nabla f(x) : y \in V_1, \lambda \in V_2 \} \subset U$ and an injective Nash-function
  \[ \phi : V_1 \rightarrow V \]
  with $\im \phi = Y \cap V$ and $\id_{V_1} = \pi \circ \phi$, where $\pi$ is the orthogonal projection onto T, the affine tangent space to $Y$ at $x$.

  Let $U_1 \subset V_1$ be the connected component containing $x$ and let $U_2 \subset V_2$ be the connected component of $0$.
  Then there are $\epsilon_-, \epsilon_+ >0$ with $U_2 = (-\epsilon_-,\epsilon_+)$.
  Now put $U' := \{ y + \lambda \nabla f(x) : y \in U_1, \lambda \in U_2 \}$.
  Since $\phi(x) = x \in U'$ we have, by continuity, that $\phi(U_1) \subset U'$.
  Hence the restriction $\wt{\phi} : U_1 \rightarrow U'$ is a well-defined, injective Nash-function with
  $\im \wt{\phi} = Y \cap U'$ and $\id_{U_1} = \pi \circ \wt{\phi}$.

  Let $\lambda_y \in \R$ with $\wt{\phi}(y) = y + \lambda_y \nabla f((x)$ for each $y \in U_1$.
  Then $\lambda_y \in (-\epsilon_-,\epsilon_+)$.
  Let $\wt{\pi}$ be the restriction of $\pi$ onto $\{z \in U' : f(z)>0\}$.

  Now $f$ assumes a constant sign $\sigma_1$ on $\{y + \lambda \nabla f(x): \epsilon_+ > \lambda > \lambda_y\}$
  and a constant sign $\sigma_2$ on $\{ y + \lambda \nabla f(x): \epsilon_- < \lambda < \lambda_y\}$.
  Since $\scp{\nabla f(x) , \nabla f(y)} \not= 0$ we have $\sigma_1 = - \sigma_2$.
  Hence a fiber of $\wt{\pi}$ is of the form $\{ y + \lambda \nabla f(x): \epsilon_- < \lambda < \lambda_y\}$ or
  $\{ y + \lambda \nabla f(x): \epsilon_+ > \lambda > \lambda_y\}$.
  In any case the fibers of $\wt{\pi}$ are connected.
  Further $\im \wt{\pi} = U_1$, $\wt{\pi}$ is open and $U_1$ is connected.
  Hence $U' \cap \{f>0\}$ is connected and by the same argument $U' \cap \{f<0\}$, too.

  Since $x \in \partial \Int A$, we get $\Int (A \cap U') \not= \emptyset$. Hence there is a $y \in A \cap U'$ with $f(y) \not= 0$.
  By changing the sign of $f$ if necessary, we can assume that $f(y) > 0$. Now
  \[ U' \cap \{f>0\} = (U' \cap \{f>0\} \cap A) \cup (U' \cap \{f>0\} \cap \overline{\R^n \setminus A}) \]
  and since $U' \cap A \cap \overline{\R^n \setminus A} = U' \cap \partial A \subset Y$ this is a disjoint union. 
  Because the first set in the union is nonempty and $U' \cap \{f>0\}$ is connected, we have $U' \cap \{f>0\} \subset A$.

  By the same argument $U' \cap \{f<0\} \subset A$ or $U' \cap \{f<0\} \cap A = \emptyset$.
  But in the first case we would have $U' \subset A$, since $A$ is closed, and hence the contradiction $x \in \Int A$.
  Hence $U' \cap A = U' \cap \{f\geq 0\}$.
  In particular, $\partial A \cap U'  = Y \cap U'$ (due to $\sigma_1 = - \sigma_2$ for each $y \in U_1$).
\end{proof}

We will mostly use this result in its homogeneous form, which we may state as follows:
\begin{cor}\label{boundaryandcomponent}
  Let $C$ be a proper cone in $\R^{n+1}$ and let $x\in\partial C$, $x\neq\vnull$. Assume that $x$ is a regular point on an irreducible component $Y$ of $\partial_a(C)$ and, moreover, that there exists an open neighborhood $U$ of $x$ in $\R^{n+1}$ such that $\partial C\cap U\subset Y(\R)\cap U$. Then there exists an open neighborhood $V$ of $x$ in $U$ satisfying
  \[
    \partial C\cap V= Y(\R)\cap V.\eqno{\Box}
  \]
\end{cor}

\begin{Thm}\label{dim_of_faces}
  Let $C$ be a proper cone in $\R^{n+1}$ and let $P$ be a patch of $C$ over an irreducible component $Y$ of $\partial_a C$.
  \begin{enumerate}
    \item $P$ is of pure projective dimension $n-1$.
    \item $\pi_1(P)$ is open in $\partial C\subset\P^n(\R)$ and of pure projective dimension $n-1$.
    \item $\pi_2(P)$ is an open semi-algebraic subset of $Y^*_{\rm reg}(\R)$. In particular, it is of pure dimension $d$ with $d=\dim(Y^*)$.
    \item For $\ell\in C^\vee$, let $H_\ell=\{x\in\R^{n+1}\suchthat \scp{\ell,x}=0\}$. Then
    \[ 
      \dim(C\cap H_\ell)=n-d
    \]
    as a cone in $\R^{n+1}$ for all $\ell\in\pi_2(P)$. Furthermore, $\pi_1(P)\cap H_\ell$ is Zariski-dense in the face $C\cap H_\ell$, for all $\ell\in\pi_2(P)$.
    \item In the situation of (4), the union of $\pi_1(P)\cap H_\ell$ taken over all patches $P$ of $C$ over $Y$ is dense in $C\cap H_\ell$ in the euclidean topology.
  \end{enumerate}
\end{Thm}

\begin{proof}
  For (1), we use \Cref{boundaryandcomponent}: For $(x,\ell)$ to be in a patch $P$ over $Y = \cV(f)$, $x$ needs to be a regular point of $Y$ and not lie on any other irreducible component of $\partial_a C$, implying $\ell = \nabla f(x)$ (see \Cref{conormalhypersurface}). So \Cref{boundaryandcomponent} shows that $(y,\nabla f(y)) \in P$ for all $y$ in $U\cap Y$ for an open neighborhood $U$ of $x$, hence $P$ has local dimension $n-1$ at $x$. This also proves (2).

  For (3), we show the following: If $U \subset \partial C$ is open (in the euclidean topology) and contained in $Y_{\rm bireg}(\R)\cap (\partial_a C)_{\rm reg}(\R)$, then $\pi_2 (\pi_1^{-1}(U))$ (via the normal cycle) is open in $Y^*_{\rm reg}(\R)$ in the euclidean topology. \\
  Note that $B := \{ (x,\ell)\in \R^{n+1} \times \R^{n+1} | \ell \in Y^*_{\rm reg} \text{ and } x \equiv 0 \text{ on } \bT_\ell(X)\}$ is a real vector bundle over $Y^*_{\rm reg}(\R)$ since $Y$ is a real variety.
  Let $y \in \Int C$. Since $U \subset Y_{\rm reg}$, we have $\pi_1^{-1}(U) = \pi_{B,1}^{-1}(U) \cap \{(x,\ell)| \scp{y,\ell}>0\}$ by the biduality theorem, where $\pi_{B,1}$ is the projection of $B$ onto the first factor.
  This implies the claim.

To show (4), pick $\ell \in \pi_2(P)\subset Y^*$. Let $F_\ell = H_\ell \cap C$ be the face of $C$ exposed by $\ell$ and let
$V_\ell = \clzar(F_\ell) = {\rm span}(F_\ell)$.
We know that $\pi_1(P)\cap H_\ell \neq \emptyset$ and since $\pi_1(P)$ is open relative to $\partial C$, it follows that $\pi_1(P)\cap H_\ell$ is
Zariski-dense in $F_\ell$.

  To finish claim (4), we show that $\dim(C\cap H_\ell) = n-d$ as a subset of $\R^{n+1}$, where $d = \dim(Y^*)$ as a projective variety. We know that $\dim(C\cap H_\ell) = \dim(V_\ell)$ so it suffices to show that $\dim(V_\ell) = n-d$.
  Let $L$ be the real annihilator of $\bT_\ell Y^*$, which has dimension $n-d$ as a linear space in $\R^{n+1}$, since $Y^*$ is a real variety, and is contained in $H_\ell$.
Note also that $L \subset Y$. We show that $L = V_\ell$ by showing that $\pi_1(P) \cap H_\ell$ is Zariski-dense in $L$.

  First, $\pi_1(P)\cap H_\ell$ is contained in $L$ because $\ell$ defines the supporting hyperplane to $C$ at $F_\ell$ and for every point in $\pi_1(P)$
the local tangent hyperplane to $Y$ is the unique supporting hyperplane to $C$ at that point (by \Cref{boundaryandcomponent}).
Since $\pi_1(P)$ is open relative to $\partial C$ and $\partial C \cap U = Y(\R)\cap U$ in an open neighborhood of any $x$ with $(x,\ell) \in P$ (again by \Cref{boundaryandcomponent}), we conclude that $\pi_1(P)$ is open relative to $Y(\R)$ (in the euclidean topology) and therefore open (and non-empty) in $L$. Therefore, $\pi_1(P)\cap L$ is Zariski-dense in $L$. This shows that the dimension of $F_\ell$ is $n-d$ as claimed.

  For (5) assume that there is an open, nonempty $U \subset F_\ell = H_\ell \cap C$ such that no patch over $Y$ intersects $U$.
Then 
\[ U \subset \bigcup_{i=2}^r Y_i \cup (Y \setminus Y_{\rm bireg}), \]
where $Y = Y_1, \dots , Y_r$ are the irreducible components of $\partial_a C$.
Since $U$ is Zariski-dense in $V_\ell$ and the union above is closed it follows that $V_\ell$ is included in the union.
Hence no patch over $Y$ intersects $F_\ell$, which contradicts our assumption. Hence the claim follows.
\end{proof}

\begin{rem}
  This theorem shows that we can determine the family of patches $\mathscr{P}(C)$ from a cylindrical algebraic decomposition of $\cNC(C)$. Theoretically speaking, the set of patches is therefore computable. Practically speaking, such computations require careful thinking and specialized approaches, already for examples of moderate dimension and complexity. We discuss some aspects of computation in more detail in \Cref{sec:computations}.
\end{rem}

The following statement shows that faces inside a patch vary continuously in a certain sense inspired by flat families in algebraic geometry. Morally speaking, if we remove a face from a patch, we may recover the full face by taking the euclidean closure of the remainder.
\begin{Cor}\label{cor:continuity}
  Under the assumptions of \Cref{dim_of_faces}, we have
  \[
    \pi_1(P)\cap H_\ell\subset\overline{\pi_1(P)\setminus H_\ell},
  \]
  unless $H_\ell=Y$.
\end{Cor}

\begin{proof}
  If $\dim \pi_1(P) \cap H_\ell = n-1$, then $H_\ell \subset Y$ and hence $Y = H_\ell$.

  Now we can assume $\dim \pi_1(P) \cap H_\ell < n-1$. Define $A$ to be the euclidean closure of  $\pi_1(P) \setminus H_\ell$.
  Suppose $\pi_1(P) \cap H_\ell$ is not contained in  $A$ and pick a point $p \in \pi_1(P) \cap H_\ell \setminus A$.
  Then $p \in \pi_1(P) \setminus A$, which is an open semi-algebraic subset of $\pi_1(P)$ and hence has
  dimension $n-1$ by \Cref{dim_of_faces}. Since we have $\pi_1(P) \setminus A \subset \pi_1(P) \cap H_\ell$ by construction, we get 
  the contradiction $\dim (\pi_1(P) \cap H_\ell) \geq n-1$.
\end{proof}

The following example illustrates that the (partial) faces in an open patch need not vary continuously in the Hausdorff metric.
\begin{exm}[The helmet]
\label{Example:DivingHelmet}
  Let $\ol{B}$ be the closed unit ball centered at the origin, $a = \sqrt{1-(7/10)^2}$ and $f = 1/2 x^2+y-7/10$ and let $K$ be the intersection of the two sets
  \begin{align*}
    & \ol{B} \cup \{ (x,y,z) \in \R^3 \suchthat y \geq 0 \text{ and } x^2+z^2=1\} \text{ and}\\
    & \{(x,y,z) \in \R^3 \suchthat z \leq a \text{ and } f(x,y,z)\leq 0 \}. 
  \end{align*}
  Figure \ref{Figure:DivingHelmet2} shows the part of the boundary of $K$ that is cut out by $f$. 
  The visible oval indicates the intersection with the sphere. Since the sphere is part of the algebraic boundary, this oval is in the singular locus of $\partial_a K$.
  Indeed, there are two open patches associated to $f$. Let $P$ denote the open patch outside of the oval.
  Consider the face in the middle, namely the face cut out by $l=(0,-10/7,0)\in K^\circ$.
  This is exactly the face at which the oval touches the upper edge and
  hence the portion of $P$ in this face is the open interval under the oval.
  In every neighborhood of this face, however, there is a face, in which $P$ has a portion below and, more importantly, a portion above the oval.
  So the ``faces'' of $P$, i.e., the portions of $P$ in respective faces, do not vary continuously in the Hausdorff metric.\EndExample
  \begin{figure}
    \begin{center}
      \includegraphics[width=4.5cm]{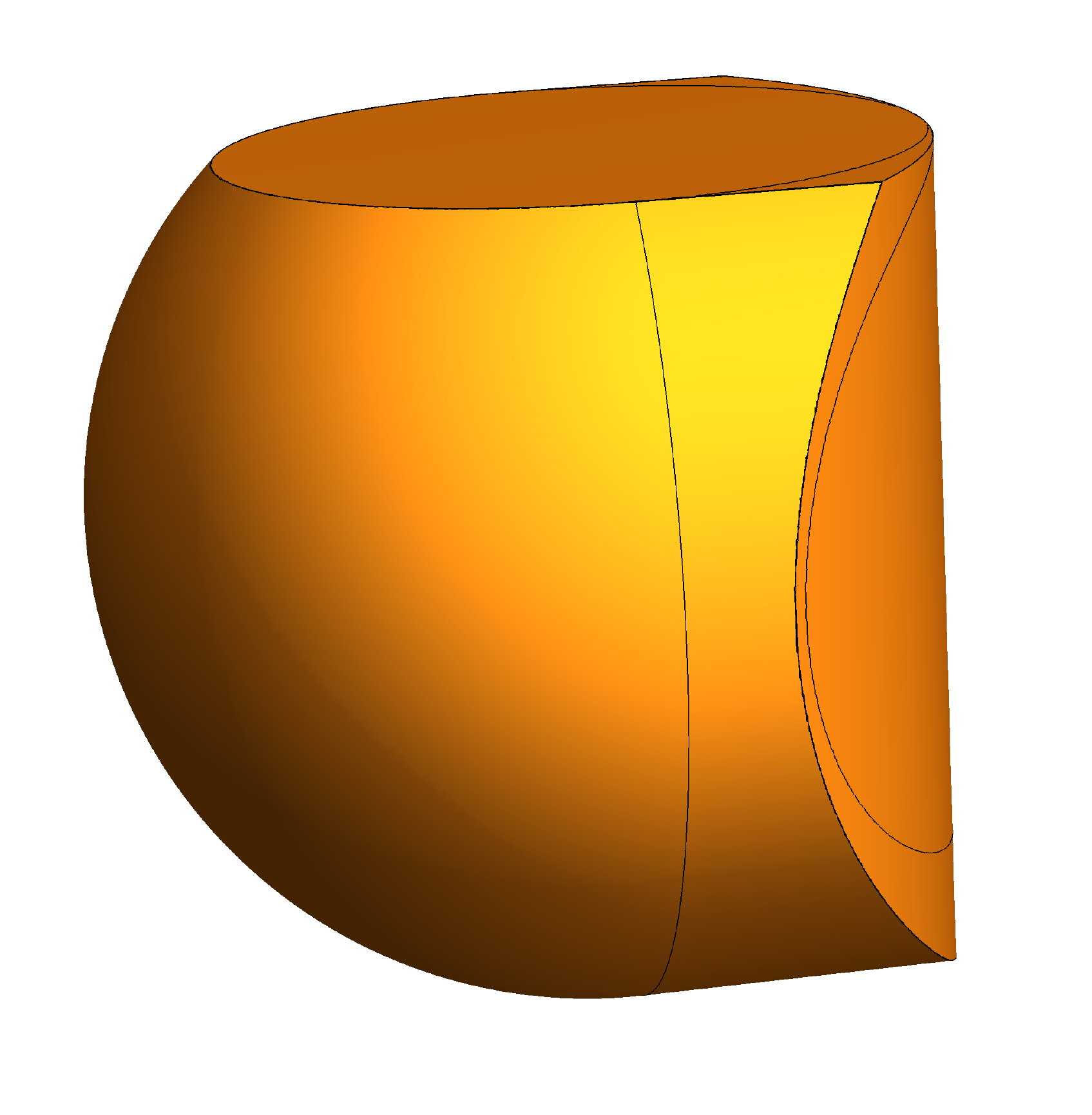}\hspace*{2em}
      \includegraphics[width=4.5cm]{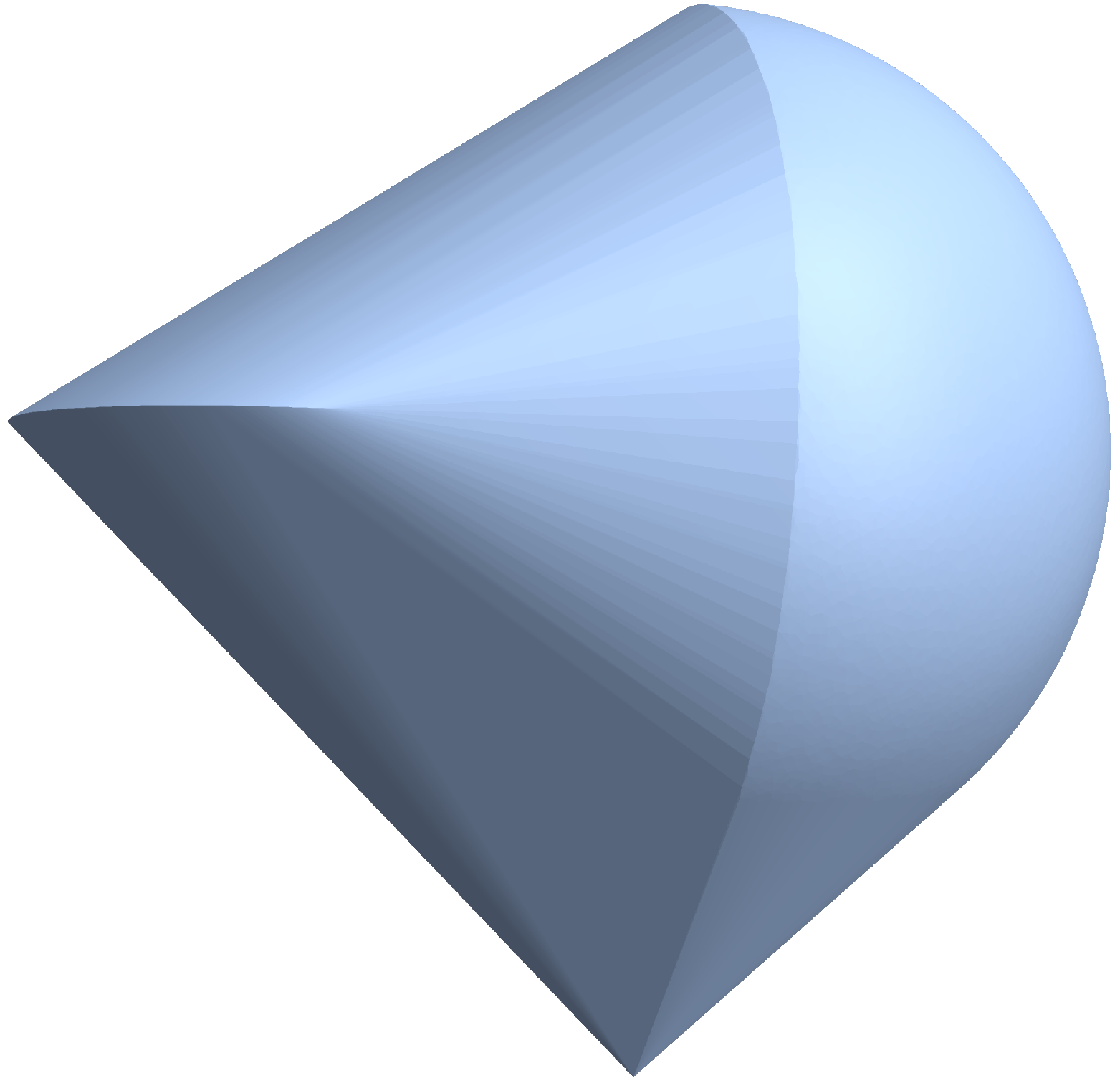}
      \caption{The ``helmet'' and its dual (\Cref{Example:DivingHelmet})}
    \end{center}
  \end{figure}
  \begin{figure}
    \begin{center}
      \includegraphics[width=5.5cm]{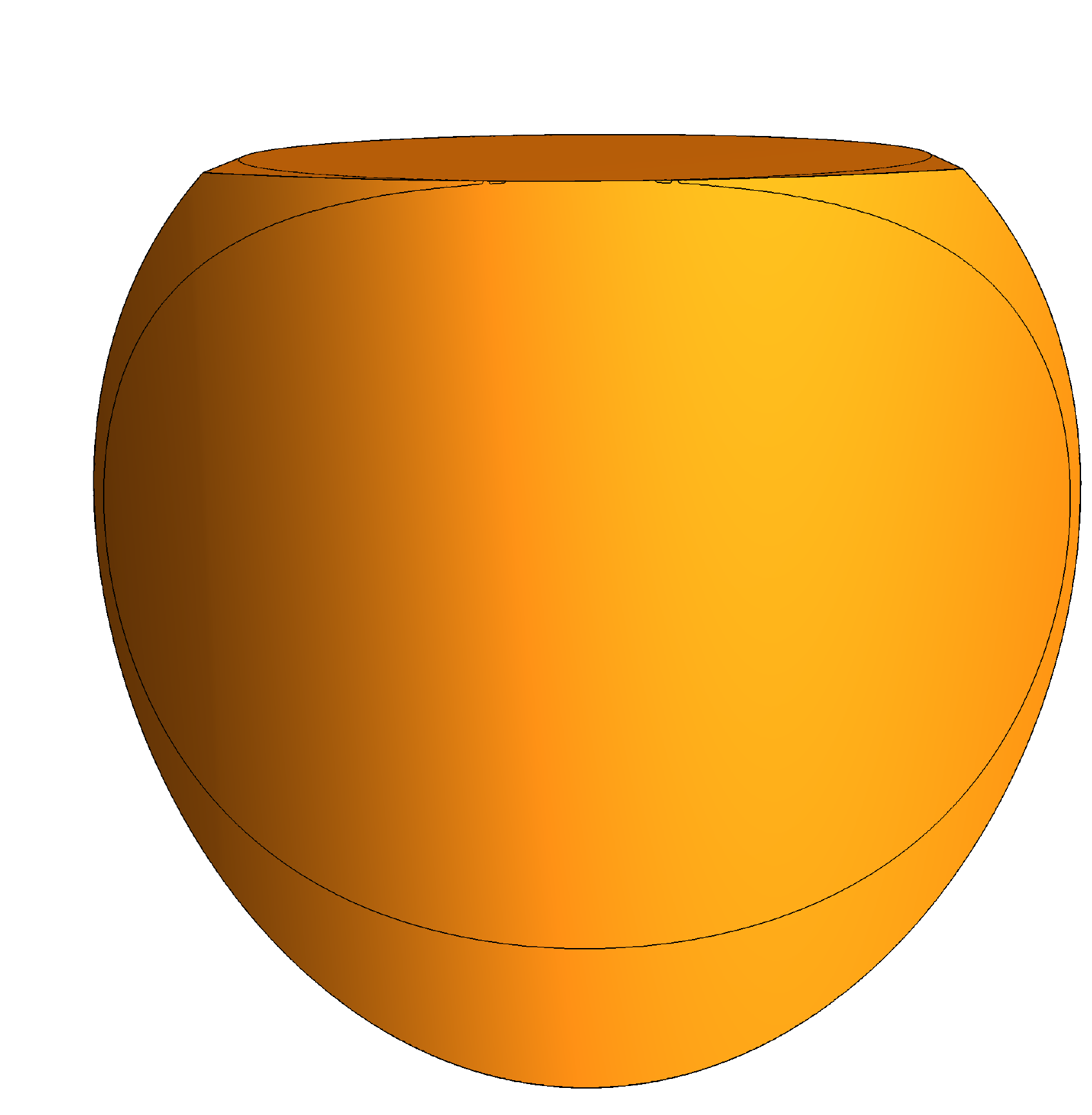}
      \caption{The ``helmet'' (\Cref{Example:DivingHelmet}) from a different viewpoint} \label{Figure:DivingHelmet2}
    \end{center}
  \end{figure}
\end{exm}
Note that the discontinuity in Example \ref{Example:DivingHelmet} disappears if we pass over to the closed patch $\ol{P}$.
In fact, this holds in general.

Since we now want to argue about Hausdorff limits, we switch to the affine setup. We discussed above in \Cref{affinesetup} what a patch means exactly here.
\begin{lem}\label{hausdorff_gen}
  Let $K \subset \R^n$ be a compact, convex, semi-algebraic set containing the origin in its interior. 
  Let $P$ be an open patch of $K$ over an irreducible component $Y$ of $\partial_a K$.
  Suppose $\ell_n \in \pi_2(P)$ is a convergent sequence whose limit $\ell$ is also in $\pi_2(P)$, and put
  $Q_n = \pi_1(\ol{P}) \cap H_{\ell_n}$ and $Q = \pi_1(\ol{P}) \cap H_\ell$.
  Then the sequence of sets $Q_n$ converges to the set $Q$ of $K$ in the Hausdorff metric.
\end{lem}
Before we prove this claim, we want to point out that the sets $Q_n$ as defined in the claim do not need to be faces of $K$ because we intersect $\pi_1(\ol{P})$ with the supporting hyperplane $H_{\ell_n}$ of $K$ and not all of the boundary of $K$. 

\begin{proof}
  For Hausdorff convergence, we need to show that both 
  \begin{align*}
    & d_n = \sup\{\inf \{\|a-b\|\colon b\in Q\} \colon a\in Q_n\}, \text{ and} \\
    & \delta_n = \sup \{ \inf \{\|a-b\|\colon a \in Q_n\} \colon b \in Q\}
  \end{align*}
  converge to zero. 
  
  Suppose $d_n$ does not converge to $0$. Then we can find an $\epsilon > 0$ and pass to a subsequence such that $d_n > \epsilon$ for all $n$.
  Hence there is a sequence $a_n \in Q_n$ such that $\inf\{\|a_n-b\|\colon b \in Q\} > \epsilon$.
  Since $\partial K$ is compact, we may pass to a convergent subsequence $a_n$. Its limit, which we call $a$, is necessarily contained in the supporting hyperplane $H_\ell$.
  By compactness of the normal cycle, the projection of the closed patch $\ol P$ is closed and therefore contains $a$, hence $a \in Q$.
  This is a contradiction because $0 < \epsilon < \inf\{\|a_n-b\|\ \colon\ b \in Q\} \leq \|a_n-a\|$ but the sequence $(a_n)$ converges to $a$.
  
  Next we show that $\delta_n$ goes to $0$ as well. Fix $\epsilon > 0$ and $b \in Q$ and let $U$ be the ball around $b$ with radius $\epsilon$. 
  We can push this neighborhood $U$ of $b$ to the dual space via the normal cycle to get the subset
  $U' = \pi_2 (\pi_1^{-1}(U) \cap P)$ of $Y^*(\R)$ which is 
  open as a subset of $Y^*(\R)$ in the euclidean topology. Since the sequence $\ell_n$ lies inside $Y^*(\R)$ and converges to $\ell$ which is in $U'$ we get that $\ell_n \in U'$ for sufficiently large $n$.
  On the primal side this condition means that $Q_n \cap U \not= \emptyset$.
  So we have $\inf\{\|a-b\|\ \colon\  a \in Q_n\} < \epsilon$ and therefore  $\delta_n \leq \epsilon$ for sufficiently large $n$ because $b$ was arbitrary.
  This proves the claim.
\end{proof}
Observe that \Cref{hausdorff_gen} is in general false if the limit $\ell$ lies in $\pi_2(\ol{P}) \setminus \pi_2(P)$,
as the face-dimension may jump up, see \Cref{exm:elliptope}.

\section{Computational Aspects}\label{sec:computations}
We discuss certain challenges for computing patches in the case of pointed convex cones $C\subset \R^{n+1}$, which means that, algebraically, we are dealing with real projective varieties $\partial_a C\subset \P^n$ and homogeneous vanishing ideals.

Our definition of a patch in \Cref{def:patch} uses connected components of the real part of the biregular locus of the conormal variety, which is difficult to compute. Let $X\subset \partial_a C\subset \P^n$ be an irreducible component of the algebraic boundary of $C$. The naive approach to computing (the Zariski closure of) $\CN(X)\setminus \CN_{\rm bireg}(X)$ is to add the ideal of the union $X_{\rm sing}\cup X^{\ast}_{\rm sing}$ of the singular loci to that of the conormal variety.

Let us return to the Cayley cubic (see \Cref{exm:elliptope}) to illustrate the computation.
\begin{Exm}
  Here, the algebraic boundary of $K$ is irreducible and equal to the Cayley cubic $\cV(f)$ for $f = x^2 + y^2 + z^2 - 2xyz - 1$. The conormal variety of the projective closure is in $\P^3\times (\P^3)^\ast$, the dual variety is Steiner's Roman Surface $\cV(X^2Y^2 + X^2Z^2 + Y^2Z^2 - 2WXYZ)$, in dual coordinates $(W,X,Y,Z)$ with $wW+xX+yY+zZ = 0$. The singular locus of the Cayley cubic consists of the four points $a,b,c,d$ shown in \Cref{exm:elliptope}. The singular locus of the Roman Surface consists of three lines, namely the line spanned by $(1:0:0:0)$ and $(0:1:0:0)$, the line spanned by $(1:0:0:0)$ and $(0:0:1:0)$, as well as the line spanned by $(1:0:0:0)$ and $(0:0:0:1)$.

  To compute $\cV(f)\setminus \cV(f)_{\rm bireg}$, we take the conormal variety of $X$ together with the above ideals and eliminate the dual coordinates. For the Cayley cubic $X$, the closure of $X\setminus X_{\rm bireg}$ is the union of $9$ lines: the six lines that are the Zariski closures of the six edges of the tetrahedron given by the four singular points on the Cayley cubic, and three lines at infinity.\EndExample
\end{Exm}

To compute the patches, we have to output descriptions of connected components of the semi-algebraic set of points in $\cNC(C) \cap \CN_{\rm bireg}(X)$ that lie over a single irreducible component $X$ of $\partial_a C$. One advantage here is that $\partial_a C$ has pure codimension $1$, so that we can use the Gauss map to the dual variety. Yet the projective dual varieties might be defective, i.e., it may have lower dimension than $\partial_a C$. This happens if the boundary of $C$ contains an open, non-empty part of points that lie in faces of dimension at least $1$, see \Cref{dim_of_faces}(4) and \Cref{Example:Bellows}.

\begin{Prop}
  Let $C\subset \R^{n+1}$ be a pointed and closed semi-algebraic convex cone and let $X$ be an irreducible component of its algebraic boundary, as above. Let $S$ be a connected component of the real part of 
  \[
    \CN_{\rm bireg}(X)\mathbin{\big\backslash} \bigcup_{Y\neq X}\CN(Y),
  \]
  where the union is taken over irreducible components $Y$ of $\partial_a C$. Hence $S$ is the set of pairs $(x,\ell)$ such that $x\in X_{\rm reg}$, $\ell\in X^*_{\rm reg}$ and $x$ lies on only one irreducible component of $\partial_a C$. Then either $S$ is a patch or $S\cap \cNC(C) = \emptyset$.
\end{Prop}

\begin{proof}
  Suppose $S\cap\cNC(C)$ is nonempty. It is a closed subset of $S$. It is also open, because for any point $(x,\ell)\in S\cap\cNC(C)$ the boundary of $C$ is locally equal to $X_{\rm bireg}(\R)$ around $x$. Thus restricting to a sufficiently small neighborhood $U$ of $x$ in $X_{\rm bireg}(\R)$,  the inverse image $\pi^{-1}(U)$ is contained in $S\cap\cNC(C)$. Since $S$ is connected, we conclude $S\subset\cNC(C)$, which shows that $S$ is a patch.
\end{proof}

In particular, one can try to compute a sample point in every connected component of the real part of the quasi-projective variety $\CN_{\rm bireg}(X)\setminus \bigcup_{Y\neq X}\CN(Y)$ and test this sample point for membership in $\cNC(C)$ to obtain a description of the patches for $C$.

\subsection{Comparison with previous work}
In \cite{patches}, the authors suggest an algorithm in Section 5 (Algorithm 5.4) to detect the number of patches for the convex hull of a sufficiently nice curve which they call simplicial. First of all, their definition of a patch is slightly different from ours. Secondly, they consider convex hulls of curves $C\subset \R^n$, and require the following properties.
\begin{enumerate}
  \item[(H1)] Every point on the curve $C$ that is in the boundary of $\conv(C)$ is an extreme point of $K = \conv(C)$.
  \item[(H2)] Every polytopal face of $K$ is a simplex.
  \item[(H3)] Every hyperplane meets the curve $C$ in finitely many points. 
\end{enumerate}

Assumptions (H2) and (H3) are easily satisfied by irreducible curves because (H2) is implied by the General Position Theorem and (H3) follows from irreducibility of the curve if it is not contained in a hyperplane. However, (H1) is an assumption that fails more often.

Their proposed algorithm \cite[Algorithm 5.4]{patches} can be viewed as a way to sample from $\NC(\conv(C))$ and the patches without knowing the algebraic boundary of $K = \conv(C)$. Starting with a finite sample of points on $C$, they first compute the convex hull of these finitely many points and record the incidences of this polytope (step 1). Facets that have normal vectors that are close to each other and intersect along a ridge are guessed to belong to a continuously varying family of facets, which is encoded in a graph on the vertex set of facet normals (step 2). The connected components of this graph are taken as proxys for the connected components of $\NC(C)\cap \CN(X)$, after a pruning subroutine (steps 4 to 10). The irreducible component $X$ of $\partial_a K$ that contains the family of faces is not computed but rather the constancy of face dimensions is substituted for local irreducibility of the boundary of $\conv(C)$.

Sampling from the primal patches of the normal cycle numerically is an interesting problem for \emph{convex hulls} of sets: In this case, it is not clear how to sample boundary points of the convex hull in order to find primal patches in the normal cycle. Sampling boundary points by optimizing in randomly chosen directions is not helpful because this samples only from dual patches. Indeed, by sampling random directions, we end up (with probability $1$ given a reasonable distribution for our sampling of directions) in subsets of $\NC(K)\subset \R^n\times (\R^n)^*$ that have dimension $n-1$ after projection on the second factor which therefore are dense in irreducible components of the algebraic boundary of the polar $K^\circ$ and this is a dual patch associated to families of faces of $K^\circ$. However, on the primal side, we cannot be sure that we cover a full-dimensional family of faces. In the example of the convex hull of a curve, we expect (with probability $1$) to have a unique optimum for a random direction and that optimum will lie on the curve so that we do not sample from the families of faces that make up the boundary of the convex hull.

To get to primal patches, we have to identify open subsets of the boundary of the convex hull itself, which could be done by stabbing the boundary with random lines --- yet it is not immediately clear in this case how to identify the face containing the sampled boundary point (which we expect to be unique with probability $1$) or even its dimension.

\section{Hyperbolicity Cones}\label{sec:hyperbolic}
Hyperbolicity cones are a well-behaved class of semi-algebraic convex cones. A polynomial $p\in \R[x_0,x_1,\ldots,x_n]$ is hyperbolic with respect to a point $e\in \R^{n+1}$ if $p(e)\neq 0$ and if for every $x\in \R^{n+1}$ the univariate polynomial $p(x+te)\in\R[t]$ is real-rooted. For every hyperbolic polynomial $p$, the connected component of $e$ in the set $\R^{n+1}\setminus \cV(p)$ is a convex cone. The facial structure at a boundary point $x$ of this cone is related to the multiplicity of the root $t=0$ of the univariate polynomials $p(x+te)$, by Renegar's work \cite{MR2198215}. In particular, the dimension of the smallest face containing $x$ (which is the unique proper face containing $x$ in case $x$ is a regular point) is determined by the rank of the Hessian matrix of $p$ at $x$, see \cite[Theorem 10]{MR2198215}. 
\begin{thm}[Renegar]
  Let $p\in\R[x_0,x_1,\ldots,x_n]$ of degree $g\geq2$  be hyperbolic with respect to $e$ and let $C\subset \R^{n+1}$ be the corresponding hyperbolicity cone. Suppose that $C$ is a proper convex cone and let $x \in \partial C$ be a regular point of $\partial_a C$. Let $d$ be the dimension of the unique proper face of $C$ containing $x$ and denote the Hessian matrix of $p$ at $x$ by $H$. Let $r$ be the rank of $H|_{T_x\cV(p)}$. In that case $r = {\rm rank}\,H -1$.
  Then $d+r = n+1$.
\end{thm}

\begin{proof}
  By \cite[Theorem 10]{MR2198215}, for every $v \in T_x \cV(p)$ we have either
  \begin{enumerate}
     \item  $p(x + tv) = 0$ for all $t \in \R$, and there exists an $\epsilon >0$ such that $x + t v \in \partial C$ for every $-\epsilon<t<\epsilon$, or
     \item $v^\top H v < 0$.
  \end{enumerate}
  This implies that $W := \{v \in T_x \cV(p) | v^\top H v = 0 \}$ is a vector space, namely the span of the face containing $x$.
  Since $H$ is orthogonally diagonalizable, we may choose an orthogonal decomposition $\R^{n+1} = V_+ \oplus V_- \oplus V_0$,
  where $v^\top H v >0$ for every $v \in V_+ \setminus \{0\}$, $v^\top H v <0$ for every $v \in V_- \setminus \{0\}$ and $V_0 = {\rm Ker}\, H$.
  Note that $ V_0 \subset T_x\cV(p)$ since $\scp{\nabla p_x, v} = v^\top \nabla p_x = \frac{1}{g-1} v^\top H x = \frac{1}{g-1} x^\top H v = 0$
  for $v \in V_0$.

  Since $V_+^\bot = V_- \oplus V_0$, we arrive at $V_0 = W \cap V_+^\bot$.
  Note that $V_+$ has at most dimension $1$.
  Now $x^\top H x =(g-1) x^\top \nabla p_x = 0$, but $H x = (g-1) \nabla p_x \not= 0$.
  Hence $x \in W \setminus V_0$. This implies ${\rm dim}\, V_+ =1$ and ${\rm dim}\, W = {\rm dim}\, V_0 + 1$.
  We conclude
  \[ d + r = {\rm dim}\, W + \dim\,T_x\cV(p) - {\rm dim}\,V_0 = n+1.\qedhere\]
\end{proof}

In the situation of the above theorem, the local dimension of the image of the Gauss map around such a point $x$ is $r$. So the Hessian matrix being of maximal rank means that the derivative of the Gauss map has maximal rank. However, having a Hessian matrix of maximal rank at a point $x \in \partial C$ does not imply that $x$ is in the biregular locus and it does not imply Hausdorff continuity of the faces either (even in the spectrahedral case) as the following examples show.

We have already observed the non-continuity in the Hausdorff metric, even for constant face dimension and spectrahedra, in \Cref{Example:Bellows}. The Hessian matrix of the two quadratic cones has rank $3$ everywhere (here, we take the $4\times 4$ Hessian matrix of the homogenization).

\begin{exm}\label{Example:Hyperb}
  Let $f \in \Q[x_0, \dots, x_3]$ be given by the definite determinantal representation $f = {\rm det}(x_0 M_0 + \dots  +x_3 M_3)$,
  where $M_0 = I$, the identity matrix, $M_1 = {\rm diag}(1,1,1,0)$,
  \begin{align*}
    M_2 &= \left(
      \begin{array}{cccc}
        0 & 1 & 0 & -3 \\
        1 & \frac{4}{9} & 0 & -\frac{4}{3} \\
        0 & 0 & -\frac{1}{4} & 1 \\
       -3 & -\frac{4}{3} & 1 & 0 \\
      \end{array}
    \right) \text{ and}\\
    M_3 &= \left(
      \begin{array}{cccc}
        0 & 0 & 0 & 0 \\
        0 & -\frac{10}{3} & 0 & 5 \\
        0 & 0 & 0 & 0 \\
        0 & 5 & 0 & 0 \\
      \end{array}
    \right).
  \end{align*}
  This calculates to:
  \begin{align*}
    f &=x_0^4+3 x_0^3 x_1+3 x_0^2 x_1^2+x_0 x_1^3+\frac{7}{36} x_0^3 x_2+\frac{7}{18} x_0^2 x_1 x_2+\frac{7}{36} x_0 x_1^2
   x_2-\frac{116}{9} x_0^2 x_2^2\\
      &-\frac{74}{3} x_0 x_1 x_2^2-\frac{106}{9} x_1^2 x_2^2+\frac{13}{2} x_0
   x_2^3+\frac{25}{4} x_1 x_2^3-\frac{10}{3} x_0^3 x_3-\frac{20}{3} x_0^2 x_1 x_3\\
      &-\frac{10}{3} x_0 x_1^2
   x_3+\frac{85}{6} x_0^2 x_2 x_3+\frac{55}{2} x_0 x_1 x_2 x_3+\frac{40}{3} x_1^2 x_2 x_3\\
      &-25 x_0^2 x_3^2-50 x_0 x_1
   x_3^2-25 x_1^2 x_3^2+\frac{25}{4} x_0 x_2 x_3^2+\frac{25}{4} x_1 x_2 x_3^2
  \end{align*}
    Then $f$ is hyperbolic with respect to $e_0$. Now $f$ is irreducible (over $\C$) and the point $p = e_1$ lies on the boundary
    of the hyperbolicity cone $C\subset \R^4$ corresponding to $(f,e)$. The Hessian matrix  at $p$ is
    \[
      \left(
      \begin{array}{cccc}
        6 & 3 & \frac{7}{36} & -\frac{1}{3} \\[0.25em]
        3 & 0 & 0 & 0\\[0.25em]
        \frac{7}{36} & 0 & -\frac{212}{9} & \frac{40}{3} \\[0.25em]
       -\frac{10}{3} & 0 & \frac{40}{3} & -50 \\[0.25em]
      \end{array}
      \right)
    \]
    and has rank $4$. Furthermore $\partial_a C = \cV(f) =: X$ and $p \in X_{\rm reg}$ with $\ell := \nabla f_p = e_0$.
    Hence, as noted above, the image of the Gauss map on an open neighborhood of $p$ in $C$ (resp.~in $X$) is a real (resp.~complex), smooth manifold of
    dimension $3$.
    If the dual variety $X^*$ were to agree locally with this complex manifold at $\ell$, then $\ell \in X^*_{\rm reg}$ by \cite[pp.\ 13f.]{Milnor1969}
    and hence $p \in X_{\rm bireg}$.
    We will see that this is not the case and that in fact $\ell \in X^*_{\rm sing}$.
    
    Consider the point $a = (0,0,1,i)$. This point satisfies
    \[ f(a) = 0 \text{ and } \nabla f_a = \frac{1}{4} e_0 .\]
    But this means that $(p,\ell) \in \CN(X)$ and $(a,\ell) \in \CN(X)$. 
    Since $p \not= a$ we have $\ell \in X^*_{\rm sing}$ by biduality.
    Alternatively one can check that $\nabla g(\ell) = 0$, where $\mathcal{I}(X^*) = \langle g \rangle$.
    
    In primitive form, $g$ has degree 8, 127 terms and \[-56099859565230000000 \; X_0^2 X_1^4 X_3^2\] is the term with the coefficient of highest
    absolute value.
    
    \Cref{Figure:Hyperb} shows the affine portion of $X_\R$ in the hyperplane $x_0 + \frac{1}{2} x_1 = 1$
    with coordinates $x = x_1$, $y = x_2$ and $z = x_3$, and
    the corresponding affine portion of $X^*_\R$ in the hyperplane $X_0 = 1$ with
    coordinates $x = X_1 - \frac{1}{2}$, $y = X_2$ and $z = X_3$.
    Note that \Cref{Figure:Hyperb} does not show components of codimension greater one and in the case of the dual does not show the
    dual convex body.
    
    In this setting, $p$ translates to the point $(2,0,0)$ and $\ell$ translates to $(-\frac{1}{2},0,0)$.
    Indeed, $X^*$ is not a smooth manifold at $\ell$, as the $x$-axis in \Cref{Figure:Hyperb} is part of $X^*$. 
  \begin{figure}
    \begin{center}
      \begin{minipage}{.5\textwidth}
      \includegraphics[width=6cm]{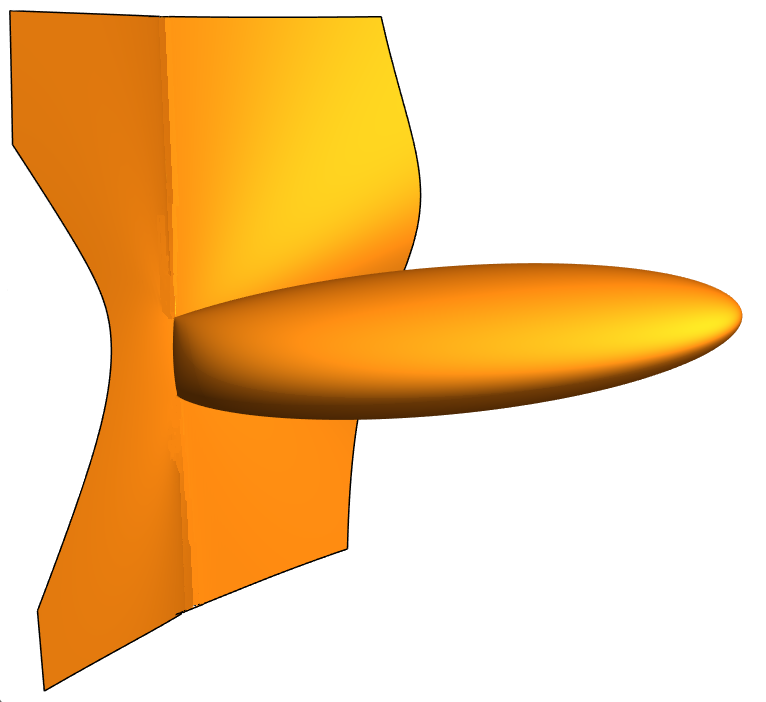}
      \end{minipage}
      \hspace*{2em}
      \begin{minipage}{.4\textwidth}
        \includegraphics[width=3.2cm]{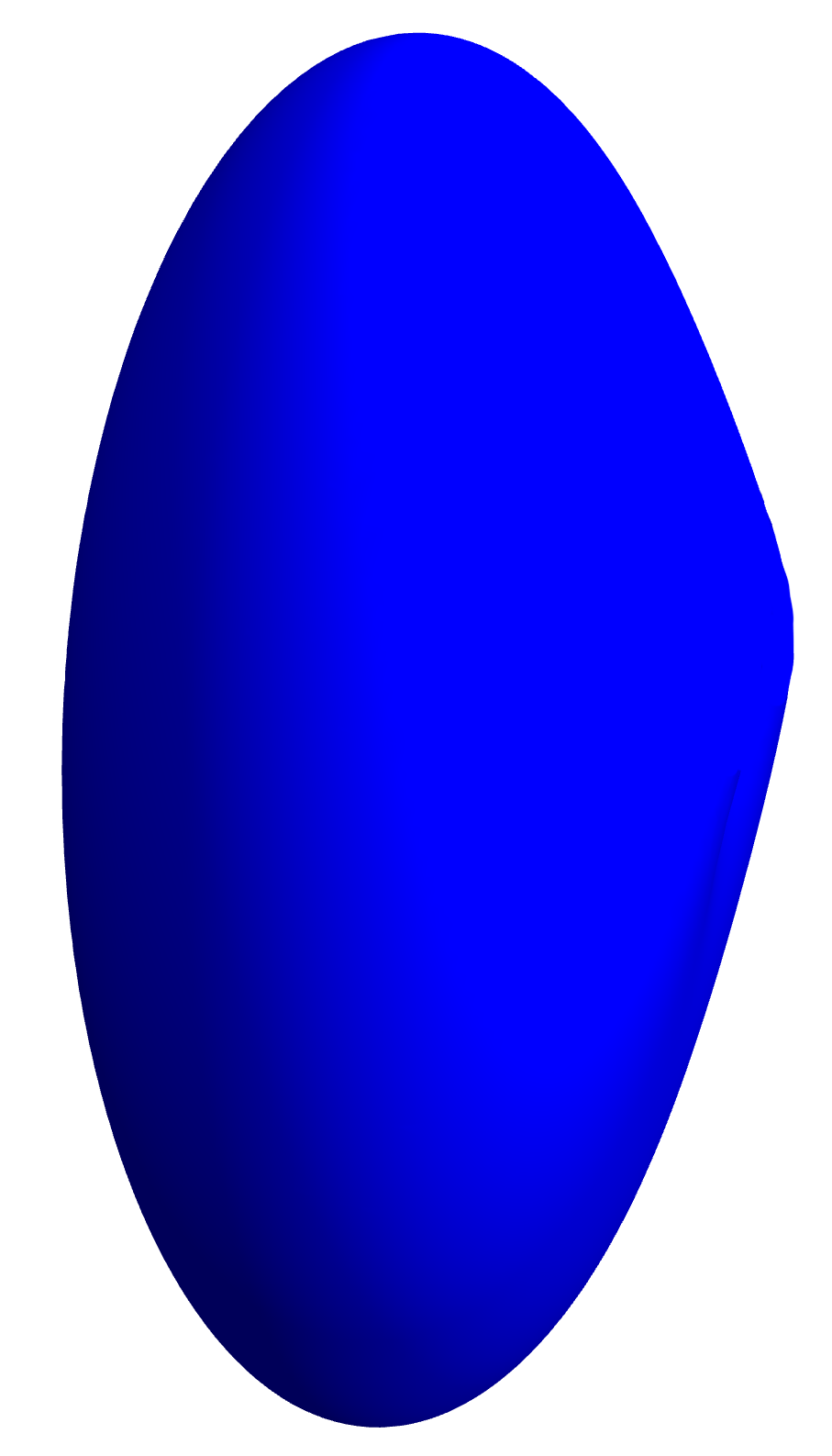}
      \end{minipage}
      \caption{An affine portion of $X_\R$ from (\Cref{Example:Hyperb}) and the respective affine portion of $X^*_\R$.} \label{Figure:Hyperb}
    \end{center}
  \end{figure}
  \EndExample
\end{exm}

For hyperbolicity cones, patches capture families of faces both in their algebraic structure as well as their convex geometric structure, see~\Cref{hausdorff_hyperb}. The main reason for this is the following special property.
\begin{lem} \label{hyperb_relint}
  Let $C$ be the hyperbolicity cone of $p\in\R[x_0,x_1,\ldots,x_n]$ in $\R^{n+1}$ and suppose that $C$ is proper. Then for every face $F$ of $C$ either there is exactly one open patch $P$ with
  ${\rm relint}\, F \subset \pi_1(P)$, or there is no open patch intersecting $F$ at all.
\end{lem}
\begin{proof}
  Let $e$ be an interior point of $C$ and write ${\rm mult}(x)$ for the multiplicity of $0$ as a root of $p(x+te)$ for $x \in \R^{n+1}$.
  Note that for a point $x \in \partial C \cap (\partial_a C)_{\rm reg}$ we have $\scp{\nabla f_x, e} > 0$ since
  $\partial_a C = \cV(p)$ and $e \in \Int C$. This implies the equivalence ${\rm mult}(x) \geq 2 \iff x \in (\partial_a C)_{\rm sing}$ for $x \in \partial C$.
  
  By \cite[Theorem 24]{MR2198215}, the multiplicity is constant on the relative interior of any face $F$.
  Hence either there is a biregular point in the interior of a face and the face is covered by exactly one open patch,
  or there is no biregular point in the interior and no open patch intersects the face.
\end{proof}

General theory of hyperbolic polynomials implies the following useful facts for homogenization:
Let $p\in\R[x_0,x_1,\ldots,x_n]$ be hyperbolic with respect to $e$ in $\R^{n+1}$ and let $C$ be its hyperbolicity cone. Suppose $C$ is pointed and let $K=\{x\in C\suchthat \scp{x,e} = \|e\|^2\}$. Then $K$ is a compact convex set containing $e$ in its interior and $C = \wh{K}$ with properly chosen coordinates. We refer to $K$ as the \emph{hyperbolic body} of $p$. For hyperbolic bodies, the (partial) faces of an open patch already vary continuously in the Hausdorff metric.

\begin{cor} \label{hausdorff_hyperb}
  Let $p\in \R[x_0,x_1,\ldots,x_n]$ be hyperbolic with respect to $e$ and let $K$ be the associated hyperbolic body.
  Let $P$ be an open patch of $K$ and 
  let $(\ell_n)_{n\in \N}\subset \pi_2(P)$ be a convergent sequence with limit $\ell$ that is also in $\pi_2(P)$. Write $H_{\ell_n}$ (and $H_\ell$ respectively) for the supporting hyperplanes corresponding to $\ell_n\in(\R^n)^*$ (and $\ell$ respectively) and put $Q_n = \pi_1(P)\cap H_{\ell_n}$ and $Q=\pi_1(P)\cap H_\ell$. 
  Then $Q_n$ converges to $Q$ in the Hausdorff metric.
\end{cor}
\begin{proof}
    Again, put $d_n = \sup \{\inf\{ \|a-b\|\colon b\in Q\} \colon a\in Q_n\}$ and, symmetrically,
  $\delta_n = \sup \{ \inf\{ \|a-b\|\colon a\in Q_n\} \colon b\in Q\}$.
  
  Showing that $\delta_n$ goes to $0$ is analogous to the argument given in the proof of \Cref{hausdorff_gen}.
  Assume for contradiction that $d_n$ does not go to $0$. As in \Cref{hausdorff_gen} we can find an $\epsilon >0$ and pass to a subsequence such that
  there is a sequence $a_n \in Q_n$ with $\inf \{ \|a_n-b\|\colon b\in Q\} > \epsilon$.
  Since $\partial K$ is compact we can choose a convergent subsequence, if necessary, so that $(a_n)_{n\in \N}$ converges to $a$, which is then necessarily in $F_\ell = \partial K \cap H_\ell$.
  Now \Cref{hyperb_relint} says that $Q = {\rm relint}(\partial K \cap H_\ell)$, which is dense in $F_\ell$.
  This leads to a contradiction because $0 < \epsilon < \inf \{ \|a_n-b\|\colon b\in Q\} \leq \|a_n-a\| + \inf \{ \|a-b\|\colon b\in Q\}
    = \|a_n-a\|$ but the sequence $(a_n)$ converges to $a$.
\end{proof}

\noindent \textbf{Acknowledgements.} This project was supported by the DFG grant \enquote{Geometry of hyperbolic polynomials} (Projektnr.~426054364).

  \end{document}